\documentclass[
	a4paper,
	10pt,
	DIV=9,
]
{scrartcl}

\usepackage[utf8]{inputenc}
\usepackage{amsfonts}
\usepackage{amssymb}
\usepackage{framed}
\usepackage{amsmath}
\usepackage[space]{grffile}
\usepackage{mathtools}
\usepackage{amsthm}
\usepackage{enumerate}
\usepackage[theoremdefs,final]{latexdev}
\usepackage[numbers]{natbib}
\usepackage{algpseudocode}
\usepackage{authblk}
\usepackage[all]{xy}
\usepackage[scaled=0.9]{helvet}

\newtheorem{algorithm}{Algorithm}

	\providecommand{\mathup}{\mathrm}

	\providecommand{\abs}[1]{\lvert#1\rvert}
	\newcommand{\R}{{\mathbb R}}
	
	\newcommand{\Z}{{\mathbb Z}}

	\newcommand{\dd}{\mathup{d}}
	\newcommand{\ud}{\,\dd}

	\DeclareMathOperator{\grad}{grad}
	\DeclareMathOperator{\divv}{div}

	\newcommand{\LieD}{\mathcal{L}}
	\newcommand{\interior}{\mathup{i}}
	
	\newcommand{\vol}{\mathup{vol}}
	
	\providecommand{\g}{\mathsf{g}}
	
	\providecommand{\norm}[1]{\lVert#1\rVert}
	\providecommand{\pair}[1]{\langle #1 \rangle}

	\providecommand{\id}{\mathup{id}}
	
	\newcommand{\Diff}{\mathrm{Diff}}
	\newcommand{\Xcal}{\mathfrak{X}}
	\newcommand{\Diffvol}{{\Diff_{\vol}}}

	\newcommand{\Dens}{\mathrm{Dens}}
	\newcommand{\Densbar}{\overline{\Dens}}

	\providecommand{\todo}[1]{}

	\def\X{\mathfrak X}

	\def\ph{\varphi}

	\def\i{^{-1}}

	\let\on=\operatorname

	\newcommand{\G}{G}
	\newcommand{\Gbar}{\bar{G}}
	\newcommand{\GI}{G^I}
	\newcommand{\nablaI}{\nabla^I}
	\newcommand{\nablaF}{\nabla^F}
	\newcommand{\dist}{d}
	\newcommand{\distF}{\bar{d}_F}
	\newcommand{\distFF}{\bar{d}_{\it FF}}
	\newcommand{\distorb}{\bar{d}_{\mathrm{Orb}}}
	\newcommand{\distdiv}{d_\mathrm{div}}
	\newcommand{\distI}{d_I}
	\newcommand{\distbar}{\bar{d}}
	\newcommand{\Gdiv}{G^{\mathrm{div}}}
	\newcommand{\GF}{\bar{G}^{F}}
	\newcommand{\GFF}{\bar{G}^{\it FF}}
	\newcommand{\Gorb}{\bar{G}^{\mathrm{Orb}}}

\author[1]{Martin Bauer\thanks{Supported by FWF-Project  ``Geometry of shape spaces and related infinite dimensional spaces'' (P246251)}}
\author[2]{Sarang Joshi}
\author[3]{Klas Modin\thanks{Supported by the Swedish Research Council and the Swedish Foundation for Strategic Research.}}
\affil[1]{Faculty of Mathematics,
University of Vienna, Email: bauer.martin@univie.ac.at}
\affil[2]{Department of Bioengineering, Scientific Computing and Imaging Institute, University of Utah, Email: sjoshi@sci.utah.edu}
\affil[3]{Department of Mathematical Sciences, Chalmers University of Technology and the University of Gothenburg, Email: klas.modin@chalmers.se}

\title{Diffeomorphic density matching by optimal information transport}

\begin{document}
\maketitle

\begin{abstract}
	We address the following problem: given two smooth densities on a manifold, find an optimal diffeomorphism that transforms one density into the other.
	Our framework builds on connections between the Fisher--Rao information metric on the space of probability densities and right-invariant metrics on the infinite-dimensional manifold of diffeomorphisms.
	This \emph{optimal information transport}, and modifications thereof, allows us to construct numerical algorithms for density matching. The algorithms are inherently more efficient than those based on optimal mass transport or diffeomorphic registration.
	Our methods have applications in medical image registration, texture mapping, image morphing, non-uniform random sampling, and mesh adaptivity.
	Some of these applications are illustrated in examples.

	\textbf{Keywords:} density matching, information geometry, Fisher--Rao metric, optimal transport, image registration, diffeomorphism groups, random sampling
 
	\textbf{MSC2010:} 58E50, 49Q10, 58E10

\end{abstract}
\tableofcontents
\listoftodos

\section{Introduction}\label{sec:intro}

In this paper we study the problem of finding diffeomorphic (bijective and smooth) transformations between two densities (volume forms) on an $n$--manifold~$M$ equipped with a Riemannian metric~$\g$ and volume form~$\vol$.
This has applications in many image analysis problems and is an extension of the classical image registration problem. 
Specific applications of density matching include: 
medical image registration~\cite{HaZhTaAn2004,ReHaPrMeTa2009,YiHoLi2009,Gorbunova2012786}; 
texture mapping in computer graphics~\cite{DoTa2010,Zh_et_al_2013}; 
image morphing techniques~\cite{ZhYa_et_al_2007,RaPeDeBe2012};
random sampling in Bayesian inference~\cite{MoMa2012,Re2013};
mesh adaptivity in computational methods for PDEs~\cite{BuHuRu2009}.
A more extensive list of applications and algorithms is given in~\cite{PaPeOu2014}.

The difference between classical image registration (cf.~\cite{Yo2010}) and density matching is the way transformations act.
In image registration, transformations act on positive scalar functions (images) by composition from the right.
In density matching, transformations act by pullback of volume forms: if the density is represented by a function $I\colon M\to \R^{+}$, the action is given by
\begin{subequations}\label{eq:pullback_transformations}
\begin{equation}\label{eq:pullback_transformations1}
	(\varphi,I)\mapsto \abs{D\varphi}I\circ\varphi ,
\end{equation}
where $\varphi\colon M\to M$ is the transformation and $\abs{D\varphi}$ is its Jacobian determinant.

When studying geometric aspects of density matching, it is convenient to use the framework of exterior calculus of differential forms.
A density is then thought of as a volume form, and the action is given by pullback:
\begin{equation}\label{eq:pullback_transformations2}
	(\varphi,\mu)\mapsto \varphi^*\mu .
\end{equation}
\end{subequations}
The volume form $\mu$ is related to $I$ by $\mu = I\vol$.
$I$ is the Radon--Nikodym derivative of $\mu$ with respect to $\vol$.
For the convenience, we use both the function and the exterior calculus point-of-view throughout the paper; the relation between functions and volume forms is always understood to be $\mu=I\vol$.

Let $\Diff(M)$ and $\Dens(M)$ respectively denote diffeomorphisms and normalized, smooth densities on~$M$.
Both $\Diff(M)$ and $\Dens(M)$ are infinite-dimensional manifolds (see \autoref{sec:setting} for details).
Let $\G$ denote a Riemannian metric on $\Diff(M)$ with corresponding distance function
\begin{equation}\label{eq:riemann_dist}
	\dist^{2}(\varphi,\psi) \coloneqq \inf_{\gamma(0)=\varphi, \gamma(1)=\psi} \int_{0}^{1} \G_{\gamma(t)}\big(\dot\gamma(t),\dot\gamma(t)\big)\ud t .
\end{equation}
Likewise, let $\Gbar$ denote a Riemannian metric on $\Dens(M)$ with distance function $\distbar$.
We are interested in special cases of the following two, generally formulated density matching problems:

\begin{problem}[Exact density matching] \label{prob:exact}
	Given $\mu_0,\mu_1\in \Dens(M)$, find $\varphi\in\Diff(M)$ minimizing $\dist(\id,\varphi)^{2}$ under the constraint 
	\begin{subequations}		
	\begin{equation}\label{eq:constraint_exact}
		\mu_1 = \varphi_{*}\mu_0 \coloneqq (\varphi^{-1})^{*}\mu_0.
	\end{equation}
	Equivalently, using intensity functions $I_0$ and $I_1$, the constraint is
	\begin{equation}\label{eq:constraint_exact_intensity}
		I_1 = \abs{D\varphi^{-1}}(I_0\circ \varphi^{-1}).
	\end{equation}
	\end{subequations}
\end{problem}

\begin{problem}[Inexact density matching] \label{prob:inexact}
	Given $\mu_0,\mu_1\in \Dens(M)$, find $\varphi\in\Diff(M)$ minimizing 
	\begin{equation}\label{eq:inexact_functional}
		E(\varphi) \coloneqq \sigma \dist^{2}(\id,\varphi) + \distbar^{2}(\varphi_*\mu_0,\mu_1)^{2}, \quad
		\sigma > 0\, .
	\end{equation}
	The first term in \eqref{eq:inexact_functional} is a \emph{regularity measure}: it penalizes irregularity of $\varphi$.
	The second term is a \emph{similarity measure}: it penalizes dissimilarity between $\varphi_*\mu_0$ and $\mu_1$.
	The parameter $\sigma$ is balancing the two criteria.
\end{problem}

There is no intrinsic choice of $\G$ and $\Gbar$; they are free to be specified and evaluated in the specific application.
The following choices are, however, typically considered.
\begin{enumerate}
	\item[(i)]\label{it:omt} For \autoref{prob:exact}, the standard choice is distance-squared \emph{optimal mass transport} (OMT), corresponding to the non-invariant $L^{2}$ metric
	\begin{equation}\label{eq:L2noninvariant_metric}
		\begin{split}
			\G_{\varphi}(U,V) &= \int_{M} \g(U,V)\mu_0, \quad U,V\in T_{\varphi}\Diff(M) \, , \\
			&= \int_M \g(u,v)\varphi_*\mu_0, \quad u\coloneqq U\circ\varphi^{-1},v\coloneqq V\circ\varphi^{-1} \, .
		\end{split}
	\end{equation}
	This choice of metric induces the Wasserstein distance on $\Dens(M)$~\cite{Ot2001,Vi2009}.

	\item[(ii)]\label{it:lddmm} For \autoref{prob:inexact}, a common choice~\cite{TrYo2005,HoTrYo2009} is the right-invariant $H^{k}_{\alpha}$ metric
	\begin{equation}\label{eq:H1sobolev}
		\G_{\varphi}(U,V) = \int_M \g\left( (1-\alpha\Delta)^{k}u,v \right)\vol\,, 
	\end{equation}
	where $\Delta$ is the Laplace--de~Rham operator lifted to vector fields (see \autoref{sec:setting} for details), and, as similarity measure, the $L^2$ norm
	\begin{equation}\label{eq:L2innerproduct}
		\distbar^2(\mu_0,\mu_1) = \norm{\mu_0-\mu_1}_{L^{2}}^{2} \coloneqq \int_{M} \abs{I_0-I_1}^2\vol\, .
	\end{equation}
	This setting is similar to \emph{Large Deformation Diffeomorphic Metric Matching}~(LDDMM)~\cite{JoMi2000,MiYo2001,TrYo2005}, but with the density action~\eqref{eq:pullback_transformations} instead of the composition action.
\end{enumerate}

Both choices (i) and (ii) are computationally challenging, as they require the numerical solution of nonlinear partial differential equations (the Monge--Ampere and EPDiff equations respectively).
See \cite{PaPeOu2014} and \cite{Yo2010} for efficient and stable implementations.

In this paper we consider metrics for \autoref{prob:exact} and \autoref{prob:inexact} that reduce the computational challenge to solving Poisson equations, allowing significantly faster, semi-explicit algorithms.
Our approach is based on connections between information geometry and geodesic equations on diffeomorphisms (cf.\ \emph{topological hydrodynamics}~\cite{ArKh1998,KhWe2009}), now outlined.

\citet*{KhLeMiPr2013} found that the degenerate \emph{divergence-metric} on $\Diff(M)$, given by
\begin{equation}\label{eq:div_metric}
	\Gdiv_{\varphi}(U,V) = \int_M\divv(u)\divv(v)\vol, 
\end{equation}
induces a non-degenerate metric on the quotient space $\Diffvol(M)\backslash\Diff(M)$ of right co-sets.
Here, $\Diffvol(M)\subset\Diff(M)$ denotes volume-preserving diffeomorphisms
\begin{equation}\label{eq:diffvol}
	\varphi\in\Diffvol(M) \iff \varphi^*\vol = \vol .	
\end{equation}
From Moser's principal bundle structure $\Diffvol(M)\backslash\Diff(M) \simeq \Dens(M)$ (see~\autoref{sub:bundle_structure} for details) it follows that the divergence-metric~\eqref{eq:div_metric} induces a metric on $\Dens(M)$.
Remarkably, this is the infinite-dimensional version of the \emph{Fisher--Rao metric} (also called \emph{Fisher's information metric}), predominant in information geometry.
It is given by
\begin{equation}\label{eq:fisher_rao_metric}
	\GF_{\mu}(\alpha,\beta) = \frac{1}{4}\int_M \frac{\alpha}{\mu}\frac{\beta}{\mu}\mu ,
\end{equation}
for tangent vectors $\alpha,\beta\in T_{\mu}\Dens(M)$. 
It can be interpreted as the Hessian of relative entropy, or information divergence.

Due to the degeneracy, the geometry of the divergence-metric is not ideal: instead one would like a non-degenerate metric on $\Diff(M)$ where the projection from $\Diff(M)$ to $\Dens(M)$ is a \emph{Riemannian submersion} (see~\autoref{sub:descending_metric} for details).
As remarked in~\mbox{\cite{KhLeMiPr2013}}, none of the standard right-invariant, Sobolev-type metrics on $\Diff(M)$ have this property.
Nevertheless, it is possible to construct them: a three-parameter family was introduced in~\cite{Mo2014}, where also a complete characterization of such metrics is given.
Here we focus on a specific member of the family, namely
\begin{equation}\label{eq:H1dot}
	\GI_{\varphi}(U,V) = \int_M \g(\Delta u,v)\vol + \lambda\sum_{i=1}^{k}\int_{M}\g(u,\xi_i)\vol\, \int_{M}\g(v,\xi_i)\vol,
\end{equation}
where $\lambda>0$, $\Delta$ is the Hodge-Laplacian lifted to vector fields, and $\xi_1,\ldots,\xi_k$ is an orthonormal basis of the harmonic vector fields on~$M$ (see~\autoref{sub:notation} for details). 
We call~$\GI$ the \emph{information metric}, as it is descends to the Fisher--Rao information metric as explained in~\autoref{sub:descending_metric}.

The descending property of~\eqref{eq:H1dot}
implies \emph{horizontal geodesics}.
Such horizontal geodesics describe optimal transformations between probability densities, which leads to
\emph{optimal information transport} (OIT)---an information theoretic analogue of OMT.
The analogue to the Wasserstein distance in OMT is the spherical Hellinger distance induced by the Fisher--Rao metric. 
The analogue to Brenier's polar factorization~\cite{Br1991}, which solves the OMT problem, is the information factorization of diffeomorphisms~\cite[\S\!~5]{Mo2014}, which solves the OIT problem.
Because of the descending property, Fisher--Rao geodesics on $\Dens(M)$ can be lifted to horizontal geodesics on $\Diff(M)$. This observation is the main ingredient of our algorithms.

We consider \autoref{prob:exact} and \autoref{prob:inexact} in two different settings, leading to different algorithms.
The first contribution, developed in \autoref{sec:density_matching}, considers a background metric that is compatible in the sense that $\mu_0=\vol$. This has applications in texture mapping, random sampling, and mesh adaptivity.
Numerical examples are given in~\autoref{sec:density_matching_examples}.
The second contribution, developed in \autoref{sec:density_matching_non_compatible},  makes no compatibility assumption but uses a slightly modified optimality condition to retain efficiency.  This has applications in image morphing and image registration.
Numerical examples are given in~\autoref{sec:density_matching_non_compatible_examples}.
The computational complexity of both algorithms is significantly lower than those based on OMT or LDDMM.

The emphasis of the paper is on mathematical foundations rather than applications.
The examples in~\autoref{sec:density_matching_examples} and~\autoref{sec:density_matching_non_compatible_examples} are therefore kept simple and illustrative.

\section{Geometric foundation}\label{sec:setting}
\subsection{Notation}\label{sub:notation}

Throughout the paper, the word ``metric'' always means ``Riemannian metric'' and ``distance'' always means ``Riemannian distance''.

Let $M$ be an $n$--dimensional orientable manifold with metric $\g$.
We refer to~$\g$ as the \emph{background metric}, to distinguish it from metrics on infinite-dimensional spaces.
The background metric $\g$ induces a volume form on~$M$, denoted $\vol_{\g}$. 
In oriented coordinates $x^1,\ldots,x^n$, the expression for $\vol_{\g}$ is
\begin{equation}\label{eq:volume_form_induced}
	\vol_{\g}=\sqrt{\abs{\g}}\ \dd x^1\wedge\dots\wedge \dd x^{n},
\end{equation}
where $\abs{\g}$ denotes the determinant of the metric tensor.
When the background metric~$\g$ is clear from the context, we write $\vol$ instead of $\vol_\g$.

The total volume of $M$ with respect to $\vol$, for now assumed to be finite, is denoted
\begin{equation}\label{eq:total_volume}
	\vol(M)\coloneqq \int_M \vol.
\end{equation}

The space of smooth, real valued function on $M$ is denoted $C^{\infty}(M)$.
The space of smooth vector fields and $k$--form on $M$ are denoted $\Xcal(M)$ and $\Omega^k(M)$ respectively.
The background metric $\g$ induces the \emph{musical isomorphism} $\flat\colon \Xcal(M)\to \Omega^1(M)$ by $u^{\flat} \coloneqq \g(u,\cdot)$.
Its inverse is denoted $\sharp$, so $u = (u^{\flat})^\sharp$.
If $\alpha\in\Omega^{n}(M)$ and $\mu$ is a volume form on $M$, then $\frac{\alpha}{\mu}$ is the function on $M$ defined by
\begin{equation}\label{eq:division_vol}
	\alpha = \left(\frac{\alpha}{\mu}\right) \mu.
\end{equation}
Notice that this is used in the definition of the Fisher--Rao metric~\eqref{eq:fisher_rao_metric}.
If $u,v\in\Xcal(M)$, we sometimes use the notation $u\cdot v$ instead of $\g(u,v)$.

A volume form $\mu$ on $M$ induces a divergence, defined via
\begin{equation}
	\LieD_u \mu = \divv_{\mu}(u) \mu\,,
\end{equation}
where $\LieD_u$ denotes the Lie derivative in direction of the vector field $u\in\Xcal(M)$.
When $\mu=\vol$, we write $\divv(u)$ instead of $\divv_{\vol}(u)$.

The gradient of a function $f$ on $M$ with respect to $\g$ is defined by
\begin{equation}\label{eq:gradient_def}
	\grad_{\g} f = (\dd f)^{\sharp},
\end{equation}
where $\dd\colon\Omega^{k}(M)\to\Omega^{k+1}(M)$ is the natural differential of forms.
Again, if $\g$ is clear from the context, we simply write $\grad f$.

Recall the Laplace operator $\Delta$ defined by $\Delta f = \divv\grad f$.
We also use $\Delta$ to denote the lifted Laplace--de~Rham operator on the space of vector fields, defined by
$\Delta u = -(\delta \dd u^{\flat} + \dd \delta u^{\flat})^{\sharp}$, where $\delta\colon\Omega^{k}(M)\to\Omega^{k-1}(M)$ is the \emph{codifferential operator} (contrary to the differential $\dd$, the codifferential $\delta$ depends on the background metric~$\g$).
We sometimes denote the Laplacian by $\Delta_{\g}$ when the dependence on $\g$ needs to be stressed.
The space of harmonic vector fields is given by
\begin{equation}\label{eq:harmonic_vf}
	\mathfrak{H}(M) = \{ \xi\in\Xcal(M); \Delta_{\g} \xi = 0 \}.
\end{equation}
If $M$ is a closed manifold (compact and without boundary), then $\mathfrak{H}(M)$ is finite-dimensional.
If $(M,\g)$ is flat, then $\Delta_{\g}$ on vector fields is the standard Laplacian on functions applied elementwise.

\subsection{Space of densities and the Fisher--Rao metric}\label{sub:fisher_metric}

The space of densities $\Dens(M)$ consists of smooth volume forms with total volume $\vol(M)$:
\begin{equation}\label{eq:densM}
	\Dens(M) = \{ \mu\in\Omega^n(M); \int_M\mu = \vol(M), \mu>0 \}.
\end{equation}
We like to think of $\Dens(M)$ as an infinite-dimensional manifold.
To make this rigorous, first observe that if $M$ is compact, then the space of top-forms $\Omega^{n}(M)$ is a Fréchet space with the topology induced by the Sobolev seminorms (see~\citet[Example~1.1.5]{Ha1982} for details).
Let $c\in\R$.
Then the set $\Omega^{n}_{c}(M) = \{\alpha\in\Omega^{n}(M); \int_M\alpha = c \}$ is a closed, affine subspace of $\Omega^{n}(M)$, and $\Dens(M)$ is an open subset of $\Omega^{n}_{\vol(M)}(M)$.
Therefore, $\Dens(M)$ is a \emph{Fr\'echet manifold} (cf.~\cite[Chapter~4]{Ha1982}).
The closure of $\Dens(M)$ is a Fr\'echet manifold with boundary, given by
\begin{equation}\label{eq:densMbar}
	\Densbar(M) = \{\alpha\in\Omega^{n}_{\vol(M)}(M); \alpha \geq 0 \}.
\end{equation}
Since $\Dens(M)$ is an open subset of a closed, affine subspace of a vector space, its tangent space at $\mu$ is given by
\begin{equation}\label{eq:TdensM}
	T_{\mu}\Dens(M) = \Omega^{n}_{0}(M) .
\end{equation}
Notice that $T_{\mu}\Dens(M)$ is independent of $\mu$, so the tangent bundle is trivial $T\Dens(M)\simeq \Dens(M)\times \Omega^{n}_{0}(M)$.

As an alternative to the Fr\'echet topology discussed here, one might work with the completion of $\Dens(M)$ in the Sobolev $H^s$ topology for a differentiability index~$s$.
This space, denoted $\Dens^{s}(M)$, then becomes a \emph{Banach manifold} (see \cite{Mo2014} for details).
The results in this paper are valid in both the Fr\'echet and the Banach category.

In the case when $M$ is not compact, an infinite-dimensional manifold structure can still be given, as discussed in \autoref{sub:inf_volume}.

Recall the Fisher--Rao metric $\GF$ on $\Dens(M)$ is given by equation~\eqref{eq:fisher_rao_metric}.
This metric is weak in the Fréchet (or Banach) topology.
Nevertheless, its geodesics are well-posed.
Indeed, the astonishing property of the Fisher--Rao metric is that its geodesics are explicit.
Following~\cite{Fr1991}, we introduce the $W$--map
\begin{equation}\label{eq:sphere_map}
	W\colon \Dens(M)\to C^{\infty}(M),\quad \mu \mapsto \sqrt{\frac{\mu}{\vol}}.
\end{equation}
The infinite-dimensional sphere $S^{\infty}(M) = \{ f\in C^\infty(M); \int_M f^2\,\vol = \vol(M) \}$ is a submanifold of $C^{\infty}(M)$ and the image of $W$ is a open subset of $S^{\infty}(M)$~\cite[Theorem~3.2]{KhLeMiPr2013}.
Indeed, if $\mu\in\Dens(M)$, then
\begin{equation}\label{eq:sphere_property}
	\vol(M)=\int_{M} \mu = \int_M \frac{\mu}{\vol}\vol = \int_M W(\mu)^2 \vol \eqqcolon \norm{W(\mu)}^{2}.
\end{equation}
Let $\alpha \in T_{\mu}\Dens(M)$ and let $p \coloneqq T_\mu W\cdot \alpha = \frac{\alpha}{2\vol}\sqrt{\frac{\vol}{\mu}}$. 
Then
\begin{equation}\label{eq:sphere_prop_contin}
	\norm{p}^{2} = \int_{M} \frac{1}{4} \left(\frac{\alpha}{\vol}\right)^2 \frac{\vol}{\mu} \vol = \frac{1}{4}\int_M \left(\frac{\alpha}{\mu} \right)^{2}\mu = \GF_{\mu}(\alpha,\alpha).
\end{equation}
We have thus showed that $W$ is an isometry between $\Dens(M)$ with the Fisher--Rao metric and an open subset of~$S^{\infty}(M)$.
Since the geodesics of the infinite-dimensional sphere $S^{\infty}(M)$ are explicitly known, we obtain the geodesics on $\Dens(M)$.
Indeed, the Fisher--Rao geodesic between $\mu_0$ and $\mu_1$ is given by
\begin{equation}\label{eq:fisher_rao_geodesics}
	[0,1]\ni t\mapsto \left( 
			\frac{\sin\left((1-t)\theta\right)}{\sin\theta}f_0 + \frac{\sin\left(t\theta\right)}{\sin\theta}f_1
		\right)^2 \vol,
		\quad \theta = \arccos\left( \frac{\pair{f_0,f_1}_{L^2}}{\vol(M)} \right)\,,
\end{equation}
where $f_0 = W(\mu_0)$ and $f_1=W(\mu_1)$.

A direct consequence of the formulae~\eqref{eq:fisher_rao_geodesics} is an explicit formula for the induced geodesic distance. 
Indeed, if $\distF$ denotes the distance function of the Fisher--Rao metric, then
\begin{equation}\label{eq:FR_distance}
	\distF(\mu_0,\mu_1)= \sqrt{\vol(M)}\arccos\left(\frac{1}{\vol(M)}\int_M \sqrt{\frac{\mu_0}{\vol}\frac{\mu_1}{\vol}}\vol\right).
\end{equation}
As already mentioned, this formula for $\distF(\mu_0,\mu_1)$ is the key ingredient for our density matching algorithms.
It is a spherical version of the Hellinger distance~\cite{KhLeMiPr2013}. In \autoref{sec:relation_between_the_wasserstein_and_fisher_rao_distances}
we compare the geodesic distance of the Fisher-Rao metric to other commonly used distance functions on the space of probability measures.
\begin{remark}
	Recall that the Fisher--Rao metric~$\GF$ on $\Dens(M)$ is canonical: it does not depend on the choice of background metric.
	For the $W$--map in equation~\eqref{eq:sphere_map}, this implies that~$\vol$ can be any volume form, as long as $\mu$ is absolutely continuous with respect to it.
	In particular, as in \autoref{ex:explicit} below, it does not have to be the volume form associated with the background metric.
\end{remark}

\begin{remark}
	In information geometry~\cite{AmNa2000}, a finite-dimensional submanifold~$\Gamma\subset\Dens(M)$ is called a \emph{statistical manifold}.
	The Fisher--Rao metric on $\Dens(M)$ induces a metric on $\Gamma$;
	in local coordinates $\theta=(\theta_1,\ldots,\theta_k)\in \R^{k}$ it is given by
	\begin{align}\label{eq:fisher_rao_classical}
		G^{\Gamma}_{ij}(\theta) &= \frac{1}{4}\int_M \frac{\partial \ln p(x,\theta)}{\partial \theta_i}\frac{\partial \ln p(x,\theta)}{\partial \theta_j} p(x,\theta)\vol(x), \\
		&= \frac{1}{4}\mathrm E\left[ \frac{\partial \ln p(x,\theta)}{\partial \theta_i}\frac{\partial \ln p(x,\theta)}{\partial \theta_j} \right] \\
		&= \int_M \frac{\partial \sqrt{p(x,\theta)}}{\partial \theta_i}\frac{\partial \sqrt{p(x,\theta)}}{\partial \theta_j} \vol(x)
	\end{align}
	where $\theta\mapsto p(\cdot,\theta)\vol \in \Dens(M)$ is the local coordinate chart.  
	The tensor field $G^{\Gamma}_{ij}(\theta)$ is the classical information matrix of \citet{Fi1922}.
\end{remark}

\subsection{Group of diffeomorphisms and the information metric}\label{sub:information_metric}
The set of diffeomorphisms on $M$ is denoted $\Diff(M)$; it consists of smooth bijective mappings $M \to M$ with smooth inverses.
This set has a natural group structure, by composition of maps.
If $M$ is compact, then $\Diff(M)$ is a \emph{Fréchet Lie group}~\cite[\S\!~I.4.6]{Ha1982}, i.e., a Fréchet manifold where the group operations are smooth mappings.
The Lie algebra of $\Diff(M)$ is given by the space~$\Xcal(M)$ of smooth vector fields (tangential if $M$ has a boundary).
There is a natural choice of an $L^2$ inner product on $\Xcal(M)$, given by
\begin{equation}\label{eq:L2_inner_xcal}
	\pair{u,v}_{L^2} \coloneqq \int_M \g(u,v)\vol.
\end{equation}
The tangent space $T_{\varphi}\Diff(M)$ consists of maps $U\colon M\to TM$ with $U\circ\varphi^{-1}\in \Xcal(M)$.

As with the space of densities, one can also chose to work in the Sobolev completion $\Diff^{s}(M)$.
For large enough $s$, the set $\Diff^{s}(M)$ is a Banach manifold.
It is, however, not a Banach Lie group, because left composition is not smooth, only continuous---an issue to be carefully addressed when deriving rigorous existence results for geodesics equations on $\Diff^{s}(M)$ (see for example~\cite{EbMa1970,Mo2014}).
The case of non-compact~$M$ is discussed in~\autoref{sub:inf_volume}.

Recall that we are interested in the information metric $\GI$ on $\Diff(M)$, defined in equation~\eqref{eq:H1dot}.
Again, this metric is weak in the Fréchet (or Banach) topology.
Nevertheless, the geodesics are well-posed: local existence and uniqueness of the geodesic equation on $\Diff^{s}(M)$ with the Banach topology is given in~\cite[\S\!~3]{Mo2014};
the result is extended to $\Diff(M)$ with the Fréchet topology by standard techniques as in~\cite{EbMa1970}.

The metric $\GI$ has the property of right-invariance: if $U,V\in T_{\varphi}\Diff(M)$ then
\begin{equation}\label{eq:right_invariance}
	\GI_{\varphi}(U,V) = \GI_{\varphi\circ\psi}(U\circ\psi,V\circ\psi), \quad \forall\,\psi\in\Diff(M).
\end{equation}
This property implies that the geodesic equation can be stated in terms of the reduced variable $u = \dot\varphi\circ\varphi^{-1} \in \Xcal(M)$, given by
\begin{equation}\label{eq:euler_arnold_eq}
	\frac{\dd}{\dd t} m + \LieD_u m + m  \divv u = 0, \quad m^{\sharp} = Au,
\end{equation}
where $A\colon \Xcal(M)\to \Xcal(M)$ is the inertial operator associated with the inner product $\GI_{\id}(\cdot,\cdot)$ on $\Xcal(M)$.
Explicitly,
\begin{equation}\label{eq:inertia_operator_explicit}
	Au = \Delta u + \lambda\sum_{i=1}^{k}\xi_i \int_M\g(u,\xi_i)\vol,
\end{equation}
where $\xi_1,\ldots,\xi_k$ is an basis for the space of harmonic vector fields $\mathfrak{H}(M)$, orthogonal with respect to the $L^2$ inner product~\eqref{eq:L2_inner_xcal}.
Since $\GI$ is a non-degenerate metric, the inertia operator $A$ is invertible.
Let us now compute its inverse.

First, it follows from the Hodge decomposition of 1--forms that the space of vector fields admits the $L^2$-orthogonal decomposition $\Xcal(M) = \mathfrak{E}(M)\oplus\mathfrak{H}(M)$, where $\mathfrak{E}(M)$ is the image of the Laplace--de~Rham operator $\Delta\colon\Xcal(M)\to\Xcal(M)$.
The inertia operator is diagonal with respect to this decomposition: if $u=w + \xi$ are the components then 
\begin{equation}\label{eq:A_diagonal}
	A(w + \xi) = \Delta w + \lambda \xi, \quad \Delta w \in \mathfrak{E}(M), \lambda\xi\in\mathfrak{H}(M).
\end{equation}
Since $\Delta$ is a automorphism on $\mathfrak{E}(M)$, the inverse $\Delta^{-1}\colon \mathfrak{E}(M)\to\mathfrak{E}(M)$ is well-defined, so
\begin{equation}\label{eq:A_inverse}
	A^{-1}(w + \xi) = \Delta^{-1}w + \frac{1}{\lambda}\xi.
\end{equation}
To compute the components $w$ and $\xi$ of $u\in\Xcal(M)$, it suffices to first compute
\begin{equation}\label{eq:xi_harmonic}
	\xi = \sum_{i=1}^{k} \pair{u,\xi_i}_{L^2}\xi_i
\end{equation}
and then set $w = u - \xi$.
We have thus computed the inverse $A^{-1}\colon \Xcal(M)\to\Xcal(M)$.

\subsection{Moser's principal bundle structure} \label{sub:bundle_structure}

Recall from~\autoref{sec:intro} that the diffeomorphism group $\Diff(M)$ acts from the right on the space of densities $\Dens(M)$ via equation~\eqref{eq:pullback_transformations}.
This action is not free: the isotropy subgroup $\Diff_{\mu}(M)$ of $\mu\in\Dens(M)$ (also called stabilizer of~$\mu$) consists those diffeomorphisms that are volume preserving with respect to~$\mu$, given by
\begin{equation}\label{eq:diffmu}
	\Diff_{\mu}(M) \coloneqq \{\varphi\in\Diff(M); \varphi^*\mu=\mu \} .
\end{equation}
The action is, however, transitive, proved by \citet{Mo1965} for $\Diff(M)$ and $\Dens(M)$ and by \citet{EbMa1970} for $\Diff^{s}(M)$ and $\Dens^{s-1}(M)$; for any pair of densities $\nu,\mu$ there exists a diffeomorphism $\ph$ such that $\ph^{*}\nu = \mu$.
If we fix a reference density $\mu \in \Dens(M)$, we can therefore identify $\Dens(M)$ with the quotient space of right co-sets
\begin{equation*}
[\ph] \coloneqq \Diff_{\mu}(M)\circ \ph \in \Diff_{\mu}(M)\backslash \Diff(M).
\end{equation*}
Indeed, if $\psi\in [\ph]$, then $\psi^*\mu = \ph^*\mu$, so the map
\begin{equation*}
	\Diff_{\mu}(M)\backslash \Diff(M) \ni [\ph] \longmapsto \ph^*\mu \in \Dens(M)
\end{equation*}
is well-defined.
It is injective, by construction, and surjective, by Moser's transitivity result.
The projection map
\begin{align*}
	\pi_{\mu}: \Diff(M)\ni \ph \longmapsto \ph^*\mu  \in \Dens(M)
\end{align*}
thereby provides a principal bundle structure
\begin{equation}\label{eq:principal_bundle_diff_dens}
\xymatrix@R3ex@C3ex{
			\Diff_{\mu}(M)\; \ar@{^{(}->}[r] & \Diff(M) \ar[d]^-{\pi_{\mu}} \\ & \Dens(M)\,.
		}
\end{equation}
That is, the preimage $\pi_{\mu}^{-1}(\lambda)$ of each $\lambda\in \Dens(M)$ is a \emph{fibre} in $\Diff(M)$, and each fibre is parameterized by the left action of $\Diff_{\mu}(M)$ on $\Diff(M)$, see \citet[\S\!~III.2.5]{Ha1982} for details.
When $\mu=\vol$, we write $\pi$ instead of $\pi_{\vol}$.

\begin{remark}
In \autoref{sub:fisher_metric} we mapped $\Dens(M)$ to the a subset of $S^{\infty}(M)$ via the $W$-mapping in equation~\eqref{eq:sphere_map}.
The reason for this map was the simple interpretation of the Fisher-Rao metric in this representation (as the sphere metric). 
Through this map, $\Diff(M)$ also acts on $S^{\infty}(M)$.
Indeed, for $f\in S^{\infty}(M)$ and $\varphi\in\Diff(M)$ the action is given by
\begin{equation}\label{half_density_action}
	f\star \varphi \coloneqq \sqrt{\abs{D\varphi}}(f\circ\varphi)\,.  
\end{equation}
As required, $W(\varphi^*\mu) = W(\mu)\star\varphi$.
\end{remark}

\subsection{Riemannian submersions and descending metrics}\label{sub:descending_metric}

In this section we show how that the information metric $\GI$ and the Fisher--Rao metric $\GF$ gives a Riemannian structure to the principal bundle~\eqref{eq:principal_bundle_diff_dens}.
This Riemannian structure is the key to our density matching algorithms in~\autoref{sec:density_matching} and~\autoref{sec:density_matching_non_compatible}.

Moser's result on transitivity implies that $\pi\colon \Diff(M)\to \Dens(M)$ is a submersion: it is smooth and its derivative is surjective at every point.
Following the work in~\cite[\S\!~4]{Mo2014}, we now show that $\pi$ is, in fact, a \emph{Riemannian submersion} with respect to $\GF$ and $\GI$.

Let $\mathcal{V}_{\varphi}\subset T_{\varphi}\Diff(M)$ denote the \emph{vertical distribution} given by the tangent spaces of the fibres of the principal bundle structure~\eqref{eq:principal_bundle_diff_dens}:
\begin{equation}\label{eq:vertical_condition}
	U\in\mathcal{V}_{\varphi} \iff T_{\varphi}\pi\cdot U = 0.
\end{equation}
The \emph{horizontal distribution} $\mathcal{H}_{\varphi}\subset T_{\varphi}\Diff(M)$ with respect to the information metric $\GI$ is given by
\begin{equation}\label{eq:horizontal_condition}
	U\in\mathcal{H}_{\varphi} \iff \GI_{\varphi}(U,V) = 0, \quad \forall\, V \in \mathcal{V}_{\varphi}.
\end{equation}
In other words, $\mathcal{H}_{\varphi}$ is the orthogonal complement of $\mathcal{V}_{\varphi}$.
The metric $\GI_{\varphi}$ \emph{descends} to $\GF$ through the principle bundle structure~\eqref{eq:principal_bundle_diff_dens}.
That is
\begin{equation}\label{eq:descending_metric}
	\GI_{\varphi}(U,V) = \GF_{\pi_{\mu}(\varphi)}(T_{\varphi}\pi_{\mu}\cdot U,T_{\varphi}\pi_{\mu}\cdot V), \quad
	\forall\, U,V\in \mathcal{H}_{\varphi}.
\end{equation}
A remarkable property of descending metrics is that initially horizontal geodesics remain horizontal: if $\varphi(t)$ is a geodesic curve and $\dot{\varphi}(0) \in \mathcal{H}_{\varphi(0)}$, then $\dot{\varphi}(t) \in \mathcal{H}_{\varphi(t)}$ for all~$t$.
Furthermore, if $\varphi(t)\in\Diff(M)$ is a horizontal geodesic curve, then $\mu(t)\coloneqq \pi(\varphi(t))$ is a geodesic curve on $\Dens(M)$.

In summary, we have the following result.

\begin{lemma}[\cite{Mo2014}]\label{lem:H1dot_descending}
	Under the identification $\Dens(M)\simeq \Diff_{\vol}(M)\backslash \Diff(M)$, the information metric $\GI$, given by~\eqref{eq:H1dot}, descends to the Fisher--Rao metric $\GF$, given by~\eqref{eq:fisher_rao_metric}, i.e., $\pi\colon(\Diff(M),\GI)\to (\Dens(M),\GF)$ is a Riemannian submersion.
	The horizontal distribution is right-invariant, given by
	\begin{equation}\label{eq:horisontal_dist_H1dot}
		\mathcal{H}_{\varphi} = \{ U\in T_{\varphi}\Diff(M); U\circ \varphi^{-1} = \grad(f), f\in C^{\infty}(M) \}\, .
	\end{equation}
\end{lemma}
\begin{remark}
In \cite{BaBrMi2015} it was shown that the Fisher--Rao metric is the unique metric on $\Dens(M)$ that is invariant under the action of the diffeomorphism group. As a consequence, any right-invariant metric on $\Diff(M)$ that descends to a metric on $\Dens(M)$ through the principal bundle structure~\eqref{eq:principal_bundle_diff_dens} descends to the Fisher--Rao metric.
\end{remark}

\begin{remark}
	The condition for a right-invariant metric to descend to right cosets is transversal to the condition for a subgroup to be totally geodesic. See~\cite{MoPeMaMc2011} for details. 
\end{remark}

\subsection{Base manifold with infinite volume} \label{sub:inf_volume}

In the setting so far we assumed that $M$ is compact and that $\vol(M)$ is finite.
In some applications it is useful to drop this assumption and consider non-compact manifolds with infinite volume, in particular $M=\R^{n}$.
By imposing decay conditions on the set of densities and diffeomorphisms, the previously described theory continues to hold, as now briefly explained.
For a detailed exposition in the case of the diffeomorphism group on $\mathbb R$ see~\cite{BaBrMi2014}.

Let $\vol$ be the volume form of a non-compact Riemannian manifold $M$. 
We introduce compactly supported densities, diffeomorphisms, and functions via
\begin{align*}
	\Dens_{c}(M)& \coloneqq \left\{ I\,\vol\in \Dens(M); \left\{x ; I(x) \neq 1\right\} \text{ has compact closure}\right\}\\
	\Diff_{c}(M)& \coloneqq \left\{\ph\in \Diff(M) ; \left\{\ph ;\ph(x) \neq x\right\} \text{ has compact closure}\right\} \\
	C^{\infty}_{c}(M)& \coloneqq \left\{f\in C^{\infty}(M) ; f\text{ has compact support}\right\}.
\end{align*}
With these decay conditions, the theory described in \autoref{sub:fisher_metric}--\autoref{sub:descending_metric} for the compact case extends to the non-compact, infinite volume case. 
In particular, Moser's Lemma is still valid.

Then the space of compactly supported densities is an open subset of a sphere with infinite radius, thus a flat space.
To see this, we slightly modify the $W$-mapping~\eqref{eq:sphere_map} to
\begin{equation}
\tilde W: \Dens(M)\rightarrow C^{\infty}_c(M),\qquad   \mu \mapsto  \sqrt{\frac{\mu}{\vol}} -1\,.
\end{equation}
Using this mapping, the formula~\eqref{eq:fisher_rao_geodesics} for geodesics on $\Dens_{c}(M)$ simplifies to
\begin{equation}\label{eq:fisher_rao_geodesics_infM}
	[0,1]\ni t\mapsto \tilde W^{-1}\left( 
			(1-t)\tilde W(\mu_0) + t \tilde W(\mu_1)\right)
\end{equation}
and the induced geodesic distance is given by the Hellinger distance.

\section{Matching with compatible background metric}\label{sec:density_matching}
In this section we derive efficient algorithms to solve \autoref{prob:exact} and \autoref{prob:inexact} with respect to the information metric \eqref{eq:H1dot} on $\Diff(M)$ and the Fisher--Rao metric~\eqref{eq:fisher_rao_metric} on $\Dens(M)$ and a background metric $\g$ on $M$ fulfilling the compatibility constraints $\vol_{\g} = \mu_0$.
This property is fulfilled in some applications of density matching, for example texture mapping, random sampling, and mesh adaptivity.

An integral component of our method is the ability to horizontally lift paths in $\Dens(M)$ to the diffeomorphism group. 
Indeed, the selection of $\GI$ on $\Diff(M)$ and $\GF$ on $\Dens(M)$ fulfils two central properties.
(i) Fisher--Rao geodesics on $\Dens(M)$ are explicit~\cite{Fr1991}, and
(ii) the metrics \eqref{eq:H1dot} and \eqref{eq:fisher_rao_metric} are related via a principal bundle structure~\cite{Mo2014}.
It is these properties that allow us to construct fast, explicit algorithms. 

\subsection{Horizontal lifting of paths of densities}\label{sub:lifting}
Given a path of densities $\mu(t)\in\Dens(M)$ we want to find a path of diffeomorphisms $\ph(t)$ that project onto $\mu(t)$ with respect to Moser's principal bundle~\eqref{eq:principal_bundle_diff_dens}, i.e.,
\begin{equation}\label{eq:projection_mu0}
	\pi_{\mu(0)}(\ph(t))= \ph(t)^*\mu(0) = \mu(t).
\end{equation}
Such a path is not unique, since we can compose any solution $\ph$ from the left with any diffeomorphisms $\psi\in \Diff_{\mu(0)}(M)$.
To address the non-uniqueness, we therefore consider paths of diffeomorphisms fulfilling~\eqref{eq:projection_mu0} while of minimal length with respect to the information metric $\GI$, given by~\eqref{eq:H1dot}.
Mathematically, this problem is formulated as follows: 

\begin{problem}[Horizontal lifting] \label{prob:horiz_lifting}
\noindent Given a path of densities $\mu\in C^1([0,1],\Dens(M))$, find a path of diffeomorphisms $\ph \in C^1([0,1],\Diff(M))$ fulfilling
\begin{align}
\ph(0) &= \on{id}\, ,\\
\ph(t)^*\mu(0) &=
\mu(t)\,, \label{eq:projection_condition}
\end{align}	
while minimizing
\begin{equation}\label{equation:density_lifting}
 \int_0^1 \GI_{\ph(t)}(\dot\ph(t),\dot\ph(t))\ud t .
\end{equation}
\end{problem}
In general there is no easy way to solve this problem.
If, however, the background metric fulfils the following compatibility condition, then~\autoref{prob:horiz_lifting} reduces to solving Poisson equations.

\begin{definition}[Compatible background metric]\label{def:compatible_metric}
	Let $\mu\in\Dens(M)$.
	The background metric $\g$ on $M$ is called \emph{compatible with $\mu$} if $\vol_{\g}=\mu$.
\end{definition}

The following result explains the advantage of having a background metric compatible with $\mu(0)$.

\begin{lemma}\label{lemma:density_lifting}
	Let the background metric $\g$ on $M$ be compatible with $\mu(0)$.
	Then there is a unique path of diffeomorphisms solving \autoref{prob:horiz_lifting}.
	This path is horizontal with respect to the information metric~\eqref{eq:H1dot}.
\end{lemma} 

\begin{proof}
To prove this statement we differentiate the equation~\eqref{eq:projection_condition} for $\mu(t)$:
\begin{equation}\label{eq:dotmu}
\begin{aligned}
	\dot\mu(t) &= \partial_t(\ph(t)^*\mu(0)) = \varphi(t)^* \LieD_{v(t)}\mu(0).
\end{aligned} 
\end{equation}
Here $v(t)\in \X(M)$ denotes the right trivialized derivative of $\ph$, i.e., $\dot\ph\circ\ph\i=v$.  
Using the compatibility condition, we get
\begin{equation}\label{eq:mudot_flow}
	\dot\mu(t) = \varphi(t)^*\LieD_{v(t)}\vol_{\g} = \divv(v(t))\circ\varphi(t)\, \mu(t) .
\end{equation}
By the Hodge-Helmholz decomposition for vector fields, we can write $v$ as
\begin{equation}
 v=\grad f+w\,,
\end{equation}
for some function $f$ and a divergence free part $w$.
Equation \eqref{eq:dotmu} only determines the divergence part of the vector field $v$, which allows us to choose $w$ freely.  
Since the Hodge-Helmholz decomposition is orthogonal with respect to the information metric $\GI$, the length of $\ph$ is minimal for $w=0$. 
This is the horizontality condition, described in \autoref{sub:descending_metric}.
That the solution is unique follows from uniqueness of the information factorization of diffeomorphisms, see~\cite[\S\!~5]{Mo2014}.
\end{proof}

From \autoref{lemma:density_lifting} we obtain an equation for the solution of the lifting problem.
\begin{theorem}\label{thm:lifting}
Under the same conditions as in \autoref{lemma:density_lifting},
the unique solution of \autoref{prob:horiz_lifting} is obtained by solving the Poisson equation
\begin{equation}\label{eq:veq}
	\begin{split}
		\Delta f(t) &= \frac{\dot\mu(t)}{\mu(t)}\circ \ph(t)^{-1}, \\
		v(t) &= \grad(f(t)), \\
		\dot\ph(t) &= v(t)\circ\ph(t), \quad \ph(0) = \id \, .
	\end{split}
\end{equation}
\end{theorem}

\begin{proof}
The horizontal bundle is generated by gradient vector fields $v = \grad f$ for some function $f$.
Thus,
\begin{equation}
	\divv(v(t))=\divv(\grad(f))=\Delta f\,.
\end{equation}
From equation~\eqref{eq:mudot_flow} we then obtain~\eqref{eq:veq}.

\end{proof}
\begin{remark}
We can use the equivariance of the Laplacian and the gradient to rewrite the differential equation \eqref{eq:veq} as
\begin{equation}
	\begin{split}
		\Delta_{\varphi(t)^*g}(f(t)\circ\varphi(t)) &= \frac{\dot\mu(t)}{\mu(t)}, \\
		\varphi(t)^*v(t) &= \grad_{\varphi^*g}(f(t)\circ\varphi(t)), \\
		\dot\varphi(t) &= v(t)\circ\varphi(t), \quad \varphi(0) = \id\,. \\
	\end{split}
\end{equation}
If we introduce $h(t)\coloneqq f(t)\circ\varphi(t), w(t)\coloneqq \ph(t)^*v(t)$ and $g(t)\coloneqq\varphi(t)^*g$ the above equations become
\begin{equation}
	\begin{split}
		\Delta_{g(t)}(h(t)) &= \frac{\dot\mu(t)}{\mu(t)}, \\
		w(t)) &=\grad_{g(t)}(h(t)), \\
		\dot\varphi(t) &= T_{\id}\ph^{-1}\cdot w(t),
		\quad \varphi(0) = \id \, . \\
	\end{split}
\end{equation}
The main difference to equation \eqref{eq:veq} is the time-dependence of the Laplacian in Poisson's equation. 
\end{remark}

A numerical algorithm for solving~\autoref{prob:horiz_lifting} is now given as follows.

\begin{algorithm}\label{alg:lifting}
Assume we have a numerical way to represent functions, vector fields, and diffeomorphisms on~$M$, and numerical methods for
(i) composing functions and vector fields with diffeomorphisms,
(ii) computing the gradient of functions, and
(iii) computing the inverse of the  Laplace operator~$\Delta$.
Given a $C^1$-path of densities $[0,1]\ni t \mapsto \mu(t)$ with $\mu(0) = \vol$, a numerical algorithm for computing discrete lifted paths $\{\varphi_k\}_{k=0}^{N}$ and $\{\varphi^{-1}_{k}\}_{k=0}^{N}$ is given as follows:
\begin{enumerate}
	\item Chose a step-size $\varepsilon = 1/N$ for some positive integer~$N$.
	Initialize $t_0=0$, $\varphi_0 = \id$, and $\varphi^{-1}_{0} = \id$.
	Set $k\leftarrow 0$.

	\item Compute $I_k = \frac{\dot \mu(t_k)}{\mu(t_k)}\circ\varphi^{-1}_{k}$ and solve the Poisson equation
	\begin{equation}
		\Delta f_k= I_k.
	\end{equation}

	\item Compute the gradient vector field $v_k = \grad f_k$.

	\item Construct approximations $\psi_k$ to $\exp(\varepsilon v_k)$ and $\psi^{-1}_k$ to $\exp(-\varepsilon v_k)$, for example 
	\begin{equation}
		\psi_k = \id + \varepsilon v_k, \quad \psi^{-1}_k = \id - \varepsilon v_k.		
	\end{equation}

	\item Update the diffeomorphisms
	\begin{equation}
		\varphi_{k+1} = \psi_k\circ\varphi_k, \quad \varphi^{-1}_{k+1} = \varphi^{-1}_k\circ\psi^{-1}_k .
	\end{equation}
	
	\item Set $k \leftarrow k+1$ and continue from step 2 unless $k=N$.
\end{enumerate}
\end{algorithm}

\subsection{Exact compatible density matching (optimal information transport)}

The special case of \autoref{prob:exact} with $\GI$ and $\GF$ for infinite-dimensional metrics and a compatible background metric gives OIT.

\begin{problem}[Optimal information transport] \label{prob:oit}
	Given $\mu_1\in \Dens(M)$, find $\varphi\in\Diff(M)$ minimizing $\distI(\id,\varphi)$ under the constraint 
	\begin{subequations}		
	\begin{equation}\label{eq:constraint_oit}
		\mu_1 = \varphi_{*}\vol.
	\end{equation}
	Equivalently, using the density function $I_1$ for $\mu_1$, the constraint is
	\begin{equation}\label{eq:constraint_exact_intensity_oit}
		I_1 = \abs{D\varphi^{-1}}.
	\end{equation}
	\end{subequations}
\end{problem}

To better conform to the horizontal lifting setup in~\autoref{sub:lifting}, which uses the pullback rather than pushforward action, we notice that if $\varphi$ is a solution to \autoref{prob:oit}, then $\varphi^{-1}$ is a solution to the same problem but with pullback in~\eqref{eq:constraint_oit} instead of pushforward.
This follows from right-invariance of $\GI$, as $\distI(\id,\varphi) = \distI(\id,\varphi^{-1})$.

The following result is a direct consequence of the information factorization of diffeomorphisms~\cite[Theorem~5.6]{Mo2014}.

\begin{theorem}\label{thm:solution_exact}
	\autoref{prob:oit} has a unique solution.
	Its inverse is given by the end-point of the solution to the lifting equations~\eqref{eq:veq} for the path $\mu(t)$ given by the Fisher--Rao geodesic~\eqref{eq:fisher_rao_geodesics} with $\mu_0=\vol$.
\end{theorem}

It follows that a numerical algorithm for \autoref{prob:oit} is given by \autoref{alg:lifting} with $\mu(t)$ as in \autoref{thm:solution_exact}.
This algorithm is demonstrated in~\autoref{sec:density_matching_examples} below.
Before that, we solve the lifting equations explicitly in the one-dimensional case.

\begin{example}\label{ex:explicit}

We want to explicitly solve the \autoref{prob:oit} in dimension one. 
Let $\mu_0=I_0\ud x$, $\mu_1= I_1 \ud x$ be two arbitrary densities on~$M=S^1\simeq \R/\Z$. 
Since $M$ is one-dimensional we could solve this problem directly, using that, up to translations, the diffeomorphism $\ph$ is determined by the matching constraint only.
We shall, however, refrain from using this fact and instead solve the lifting equations~\eqref{eq:veq}.

The standard metric on $S^1$ is not compatible with $\mu_0$ unless $f_0 \equiv 1$.
Nevertheless, 
it is straightforward to construct a compatible background metric: choose $\g = I_0^2\ud x\otimes \dd x$.
Then $\vol_{\g} = \mu_0$ as required.
In contrast to the higher-dimensional case, this choice of a compatible background metric is uniquely determined by~$I_0$;
two different ways to construct compatible metrics in the higher-dimensional case are described in \autoref{sec:choice_of_background_metric}. 

Using \autoref{thm:solution_exact} we are now able to obtain the solution.
To simplify the notation, let $f_0= \sqrt{I_0}$ and $f_1=\sqrt{I_1}$, corresponding to the $W$--map $W(\mu) = \sqrt{\mu/\dd x}$.
First we recall the formula~\eqref{eq:fisher_rao_geodesics} for the Fisher--Rao geodesics on $\Dens(S^1)$, i.e.,
\begin{equation}
\mu(t) =  \left( 
			\frac{\sin\left((1-t)\theta\right)}{\sin\theta}f_0 + \frac{\sin\left(t\theta\right)}{\sin\theta}f_1
		\right)^2 \dd x \,,
\end{equation}
where the angle $\theta$ is given by
\begin{equation}
\theta = \arccos\left( \frac{1}{2\pi} \int_{0}^{2\pi} f_0 f_1 \ud x \right) \,.
\end{equation}
To calculate the lifting equations we need a formula for the gradient and Laplacian of the metric $\g$. 
For any function $f$ we have
 \begin{equation}
\grad_{\g}(f)= \frac1{I_0} \partial_x f, \qquad \Delta_g f=\frac1{f^2_0} \partial_x (f_0 \partial_x f)
\end{equation}
These formulas are derived in a more general setting in \autoref{conformalmetric}.
To simplify the notation we let
\begin{equation}
h(t,x)=\frac{\dot\mu(t)}{\mu(t)}=\frac{2 \theta  \left(\cos(\theta -t \theta) f_0(x)-\cos(t \theta) f_1(x)\right)}{ f_1(x) \sin(t \theta)+\sin(\theta -t \theta) f_0(x)}\,.
\end{equation}
The lifting equations~\eqref{eq:veq} now becomes
	\begin{equation}
		\begin{split}
			\ph_t\circ\ph^{-1}  &= \frac1{f_0}  f_x, \quad \ph(0) = \on{id}, \\
			\frac1{f_0^2} \partial_x (f_0 f_x) &= h\circ \ph^{-1}\,.
		\end{split}
	\end{equation}
The solution to this PDE is given by
\begin{equation}
\ph(t,x)=\psi(t,\psi(0,x)^{-1})\quad\text{with}\quad \psi(t,x)=\int_0^x \frac{\mu(t)}{\dd x} \dd y\,.
\end{equation}
Evaluating at $t=1$ we obtain
\begin{equation}
\ph(1,x)= \left(\int_0^x f_1 \ud y\right)\circ \left(\int_0^x f_0 \ud y\right )^{-1}\,.
\end{equation}
\end{example}

\subsection{Inexact compatible density matching}

We are now interested in the special case of \autoref{prob:inexact} with $\GI$ and $\GF$ for infinite-dimensional metrics and a compatible background metric.

\begin{problem}[Inexact, compatible density matching] \label{prob:oit_inexact}
	Given $\mu_1\in \Dens(M)$, find $\varphi\in\Diff(M)$ minimizing 
	\begin{equation}\label{eq:oit_inexact}
		E(\varphi) = \sigma \distI^2(\id,\varphi) + \distF^2(\varphi_*\vol, \mu_1).
	\end{equation}
	where $\sigma > 0$ is a fixed parameter.
\end{problem}

As in \autoref{prob:oit}, $\varphi$ is a minimizer of $E(\varphi)$ if and only if $\phi = \varphi^{-1}$ is a minimizer of
\begin{equation}\label{eq:oit_inexact_pullback}
	\tilde E(\phi) = \sigma \distI^2(\id,\phi) + \distF^2(\phi^*\vol, \mu_1).
\end{equation}
Our approach for \autoref{prob:oit_inexact} is to minimize~\eqref{eq:oit_inexact_pullback} and then obtain the solution by taking the inverse.
From \autoref{lem:H1dot_descending}, i.e., that $\GI$ descends to $\GF$, it follows that the lifting equations can be used also to obtain the solution to~\autoref{prob:oit_inexact}.

\begin{theorem}\label{thm:inexact_compatible}
	\autoref{prob:oit_inexact} has a unique solution, obtained as follows.
	Let $\ph(t)$ be the horizontally lifted geodesic such that $\ph(1)^{-1}$ solves \autoref{prob:oit} for the same~$\mu_1$.
	Let $s = \frac{1}{1+\sigma}$.
	Then the solution is given by $\ph(s)^{-1}$.
\end{theorem}

\begin{proof}
	A minimizer $\phi$ of~\eqref{eq:oit_inexact_pullback} must be connected to the identity by a horizontal geodesic (otherwise it would be possible to find another diffeomorphism with a strictly smaller value of $\tilde E$, using \cite[Theorem~5.6]{Mo2014}).
	Therefore, minimizing $\tilde E$ is equivalent to minimizing the functional $\tilde e(\mu) \coloneqq \sigma \distF^{2}(\vol,\mu) + \distF^{2}(\mu,\mu_1)$ on $\Dens(M)$.
	First, notice that $\tilde e$ is convex on $\Dens(M)$, so a unique minimizer exists.
	Denote it~$\nu$.
	From the spherical geometry of $(\Dens(M),\GF)$, explained in~\autoref{sub:information_metric}, it is cleat that $\nu$ must belong to the geodesic curve $\mu(t)$ between $\vol$ and $\mu_1$.
	Without loss of generality, we may assume that the distance between $\vol$ and $\mu_1$ is~$1$.
	Using arclength parametrization $\mu(s)$, we then get $\tilde e(\mu(s)) = \sigma s^2 + (1-s)^2$. 
	Minimization over~$s$ now proves the result.
\end{proof}

From \autoref{thm:inexact_compatible} it follows that a numerical algorithm for \autoref{prob:oit_inexact} is obtained through \autoref{alg:lifting} by solving the lifting equations until reaching $t = 1/(\sigma + 1)$ and then take the inverse.

\section{Examples of matching with compatible metric}\label{sec:density_matching_examples}

In this section give some examples of matching with a compatible metric, using \autoref{alg:lifting}.

\subsection{Random sampling from non-uniform probability distributions}

In $\R$, a classical algorithm for generating random samples from an arbitrary probability density function is to use the result in \autoref{ex:explicit}.
That is, one uses the cumulative distribution function to transform the standard uniform random variable on the unit interval. 
\autoref{alg:lifting} can analogously be used to transform uniform random samples to samples from an arbitrary probability density on~$M$. 

As an example, let $M=\mathbb{T}^2$ (the two-dimensional flat torus).
We want to produce random samples from an arbitrary probability distribution on $\mathbb{T}^2$, for example
\begin{equation}\label{eq:non_uniform_pdf}
	\mu_1 = \big(1 - 0.8 \cos (x)\cos(2 y)\big) \underbrace{\dd x\wedge\dd y}_{\vol}
\end{equation}
where $x,y\in [-\pi,\pi)$ are coordinates.
The approach is to first use the lifting algorithm for \autoref{prob:oit} to compute $\varphi\in\Diff(\mathbb{T}^{2})$ such that $\varphi_*\vol = \mu_1$, then draw random samples $(x_i,y_i)$ from the uniform distribution (using a uniform random number generator), then map these samples into the $\mu_1$-distribution by $(\tilde x_i,\tilde y_i) = \varphi(x_i,y_i)$.

For a $1024\times 1024$ grid on $\mathbb{T}^{2}$, with step-size $\varepsilon = 0.05$ (a total of $20$~steps), we obtain the following warp $\varphi$ and Jacobian $\abs{D\varphi}$.  

\begin{center}
\begin{tabular}{@{}cc@{}}
	\small Warp & \small{Jacobian $\abs{D\varphi}$} \\
	\includegraphics[height=.4\textwidth]{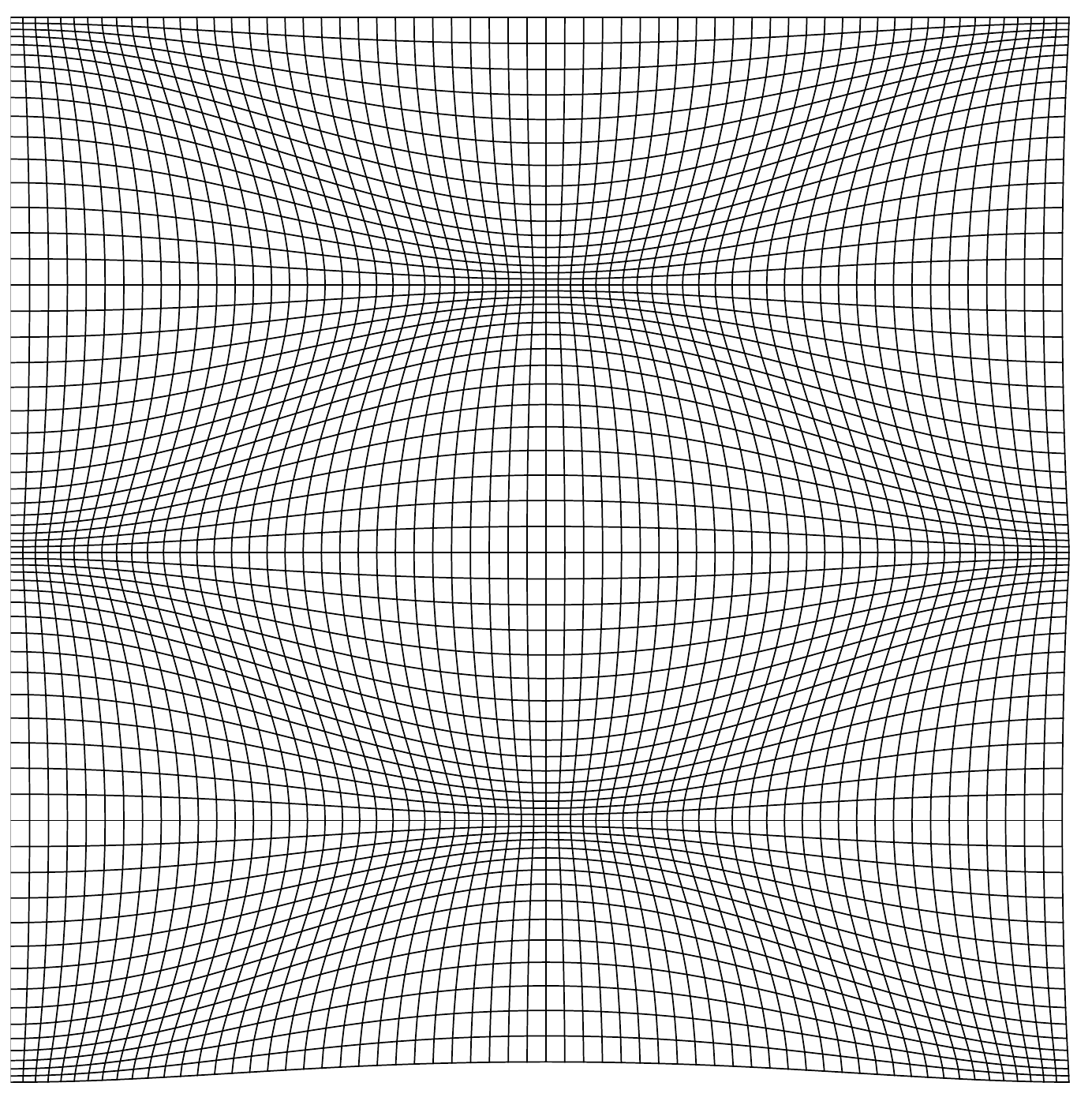}\,
	&	
	\includegraphics[height=.4\textwidth]{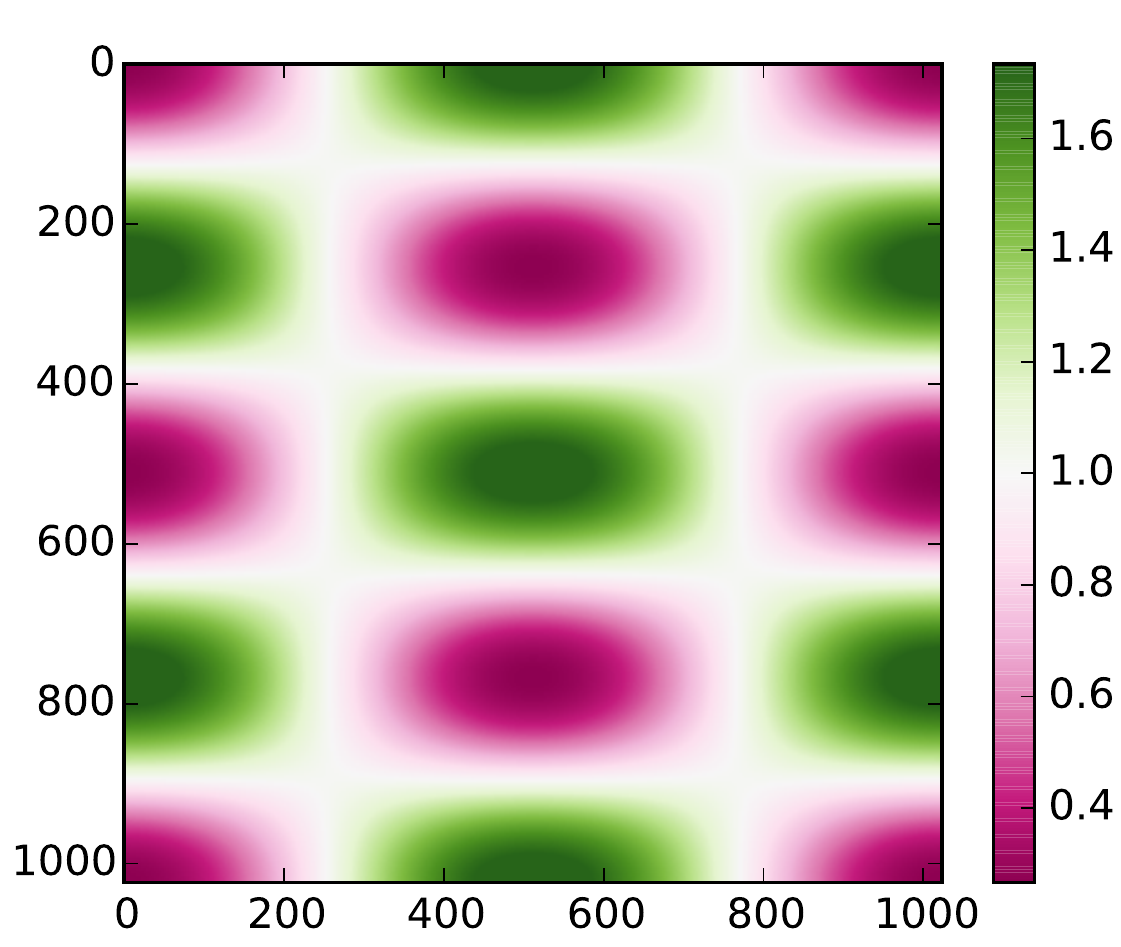}
\end{tabular}
\end{center}
\noindent
Green and pink shades of the Jacobian implies, respectively, expansion and contraction.
The optimal information framework assures the warp is matching the two probability densities while locally minimizing metric distortion.

We now draw $10^5$ uniform samples and transform them with the computed~$\varphi$.
\begin{center}
\begin{tabular}{@{}cc@{}}
	\small Uniform samples & \small{Non-uniform samples} \\
	\includegraphics[height=.42\textwidth]{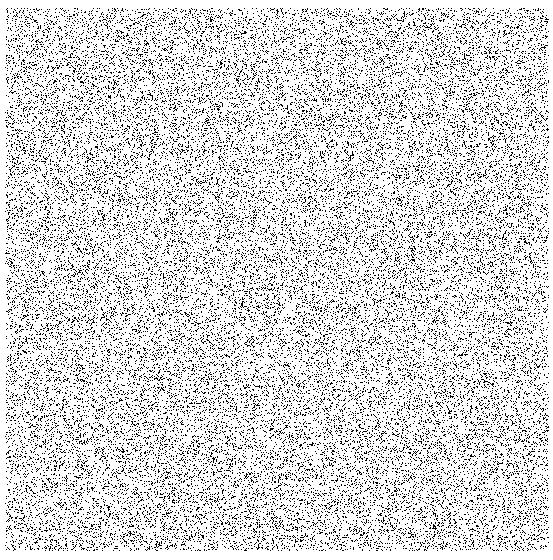}\,
	&	
	\includegraphics[height=.42\textwidth]{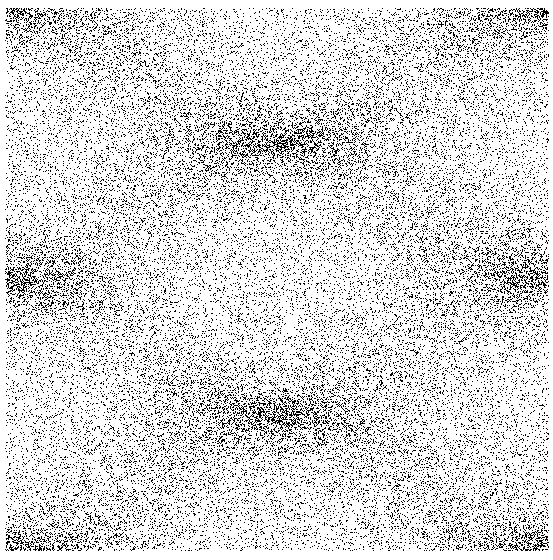}
\end{tabular}
\end{center}
\noindent
A benefit of transport-based methods over traditional Markov Chain Monte Carlo methods is cheap computation of additional samples;
it amount to drawing uniform samples and then evaluating the transformation.
A downside is poor scaling with increasing dimensions.

Detailed comparisons of the OIT approach, developed here, to the OMT approach, developed by~\citet{MoMa2012} and \citet{Re2013}, is beyond the scope of this paper.  Since OIT is intrinsically simpler than OMT (Poisson instead of Monge--Ampere equation), we expect OIT to outperform OMT based approaches.

\subsection{Registration of letters: J to V}\label{ex:compatible:imaging}

This example illustrates and explains why OIT, i.e., \autoref{prob:oit}, is \emph{not} suitable for image registration.
Recall that the algorithm developed for \autoref{prob:oit} works for any background metric.
Thus, given a source density $\mu_0\in\Dens(M)$, we can \emph{construct} a background metric $\g$ such that~$\vol_{\g} = \mu_0$ (various ways of doing this are explained in~\autoref{sec:choice_of_background_metric}).
Suppose now we want to match the letters J and V, represented as gray-scale functions~$I_0$ and~$I_1$ on $M=\mathbb{T}^{2}$.
We might have to add some background density for black pixels, since $I_0$ and $I_1$ must be strictly positive in order for \autoref{alg:lifting} to be well-defined.
Then we construct the \emph{conformal background metric} $\g$ such that $\vol_{\g} = \mu_0$.
With step-size $\varepsilon = 0.05$ ($20$~steps) and background density~$0.2$ (lowest grey-scale value, white corresponds to~$1.0$),
we get the following sequence of warps.

\begin{center}
\begin{tabular}{@{}c@{}c@{}c@{}c@{}c@{}c@{}}
	\includegraphics[width=.16\textwidth]{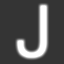}\,
	&
	\includegraphics[width=.16\textwidth]{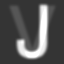}\, 
	&
	\includegraphics[width=.16\textwidth]{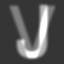}\, 
	&
	\includegraphics[width=.16\textwidth]{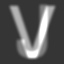}\, 
	&
	\includegraphics[width=.16\textwidth]{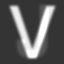}\,
	&
	\includegraphics[width=.16\textwidth]{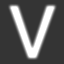}
	\\[-1ex]
	\small{Source} & & & & & \small{Target}
\end{tabular}
\end{center}

\noindent
This is simply blending between the images, as foreseen by the formula~\eqref{eq:fisher_rao_geodesics} for Fisher--Rao geodesics.
The corresponding mesh deformation and Jacobian of the inverse at the final point look as follows.
\begin{center}
\begin{tabular}{@{}cc@{}}
	\small Inverse warp & \small Inverse Jacobian $\abs{D\varphi^{-1}}$ \\
	\includegraphics[width=.45\textwidth]{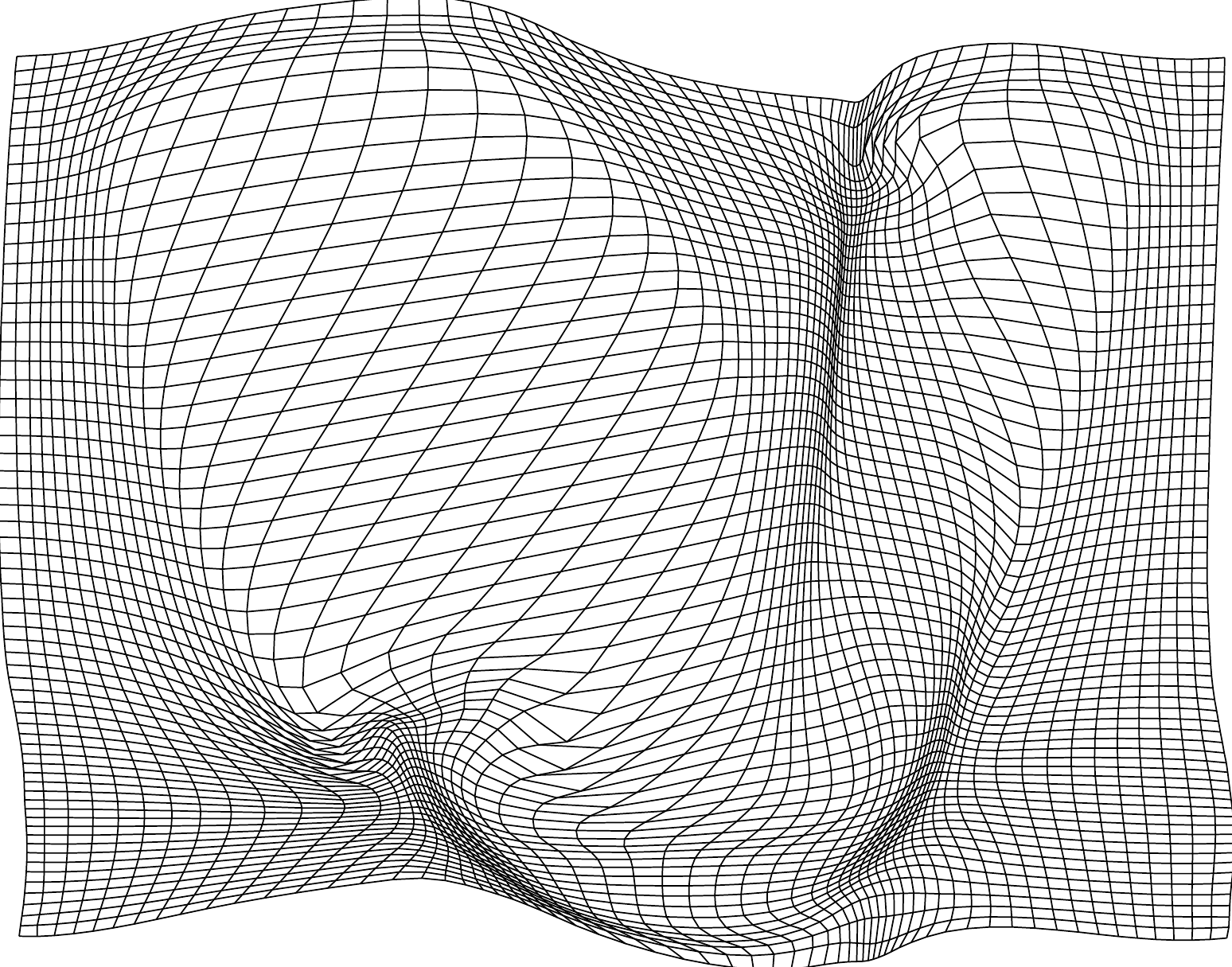} &
	\includegraphics[width=.45\textwidth]{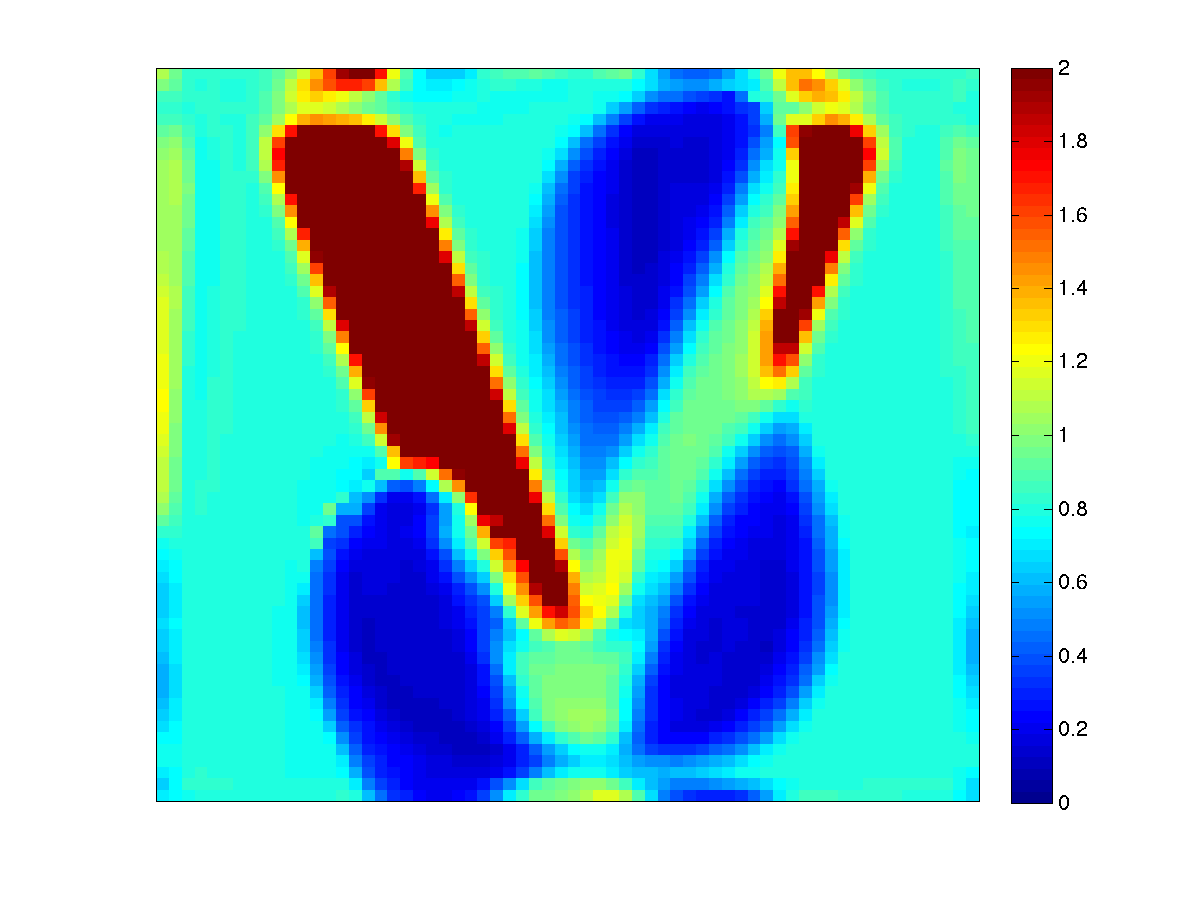} 
\end{tabular}
\end{center}
\noindent
This is \emph{not} a satisfactory registration: instead of transporting the white pixels of the J to the white pixels of the V, the resulting diffeomorphism produces white pixels by compressing the background pixels.
This example shows that, although \autoref{prob:oit} allows matching of any pair of densities by using a compatible metric, only those applications where $\mu_0$ is the standard volume form are likely to be of interest.
A remedy for more general, non-compatible matching applications is developed next.

\section{Matching with non-compatible background metric}\label{sec:density_matching_non_compatible}

In \autoref{sec:density_matching} we derived an algorithm for solving \autoref{prob:exact} and \autoref{prob:inexact} for the case when the background metric $\g$ is compatible with $\mu_0$.
In this section we want to derive an algorithm for the situation of a non-compatible background metric. 
When the background metric $\g$ is non-compatible, the solution to \autoref{prob:inexact} with respect to the information metric~\eqref{eq:H1dot} on $\Diff(M)$ and Fisher--Rao metric~\eqref{eq:fisher_rao_metric} on $\Dens(M)$ is still obtained by a geodesic curve $\ph(t)$.
However, $\ph(t)$ is not horizontal and therefore does not project to a geodesic on the space of densities.
As a consequence, the main ingredient of our efficient lifting algorithm---the explicit formula for geodesics on $\Dens(M)$---cannot be used. 
From a geometric standpoint, the problem is that the information metric $\GI$ does not descend to $\Diff_{\mu_0}(M)\backslash \Diff(M)$ unless $\mu_0=\vol$.

To numerically solve the density matching problem using the LDDMM techniques developed in~\cite{BeMiTrYo2005} is plausible, but computationally expensive. 
In the following, we shall instead study a slightly modified matching problem, for which efficient algorithms can still be obtained.

\subsection{Inexact matching with the divergence-metric.}

The modification resides in exchanging the information metric $\GI$ for the degenerate divergence-metric $\Gdiv$, given in equation~\eqref{eq:div_metric}. 
The degeneracy of the divergence-metric is characterized by 
\begin{equation}\label{eq:divv_kernel}
	\divv(U\circ\varphi^{-1})=0 \iff \Gdiv_{\varphi}(U,U) = 0,
\end{equation}
so the kernel is given by the tangent directions of the fibres of the principal bundle~\eqref{eq:principal_bundle_diff_dens}, explained in \autoref{sub:bundle_structure}.
As mentioned in the introduction, the divergence-metric descends to the Fisher--Rao metric.
If $\distdiv$ denotes the distance function of~$\Gdiv$, we have
\begin{equation}\label{eq:descend_metric_degenerate}
	\distdiv(\id,\varphi) = \distF(\vol,\varphi^*\vol) = \distF(\varphi_*\vol,\vol).
\end{equation}
In consequence, the inexact density matching problem with $\Gdiv$ and $\GF$ is the following.
\begin{problem}[Inexact matching with divergence-metric]\label{prob:inexact_Gdiv}
	Given $\mu_0,\mu_1\in \Densbar(M)$, find $\varphi\in\Diff(M)$ minimizing 
	\begin{equation}\label{eq:two_component_matching}
		E(\varphi) = E(\varphi;\mu_0,\mu_1) = \sigma \distF^2(\varphi_*\vol,\vol) + \distF^2(\varphi_*\mu_0,\mu_1),
	\end{equation}
	where $\sigma > 0$ is a regularization parameter, penalizing change of volume.
\end{problem}

Notice that we allow the source and target densities to belong to the completion $\Densbar(M)$.
This relaxation is possible because of equation~\eqref{eq:descend_metric_degenerate} and the fact that the action of $\Diff(M)$ extends naturally from $\Dens(M)$ to $\Densbar(M)$.
For applications of \autoref{prob:inexact_Gdiv} the relaxation is important, as it allows us to treat images as densities.
This only works for inexact matching, since Moser's lemma requires strictly positive densities.

Due to the degeneracy of the divergence-metric, a solution to \autoref{prob:inexact_Gdiv} is not unique.
Indeed, with 
\begin{align}
	\Diff_{\vol,\mu_0}(M) &= \Diff_{\vol}(M)\cap\Diff_{\mu_0}(M), \\
	\Diff_{\vol,\mu_1}(M) &= \Diff_{\vol}(M)\cap\Diff_{\mu_1}(M),
\end{align}
we have the following result. 

\begin{lemma}\label{lem:degeneracy}
Let $\varphi\in \Diff(M)$, $\eta_0\in \Diff_{\vol,\mu_0}(M)$, and $\eta_1\in \Diff_{\vol,\mu_1}(M)$.	
Then
\begin{align}
	E(\eta_1\circ\varphi)=E(\varphi)=E(\varphi\circ\eta_0)\,.
\end{align}
\end{lemma}
\begin{proof} 
Since $\eta_1\in \Diff_{\vol,\mu_1}(M)$ we also have that $\eta_1^{-1}\in \Diff_{\vol,\mu_1}(M)$, since the intersection of two groups is again a group. 
Using the invariance of the Fisher--Rao metric we get
\begin{align*}
	E(\eta_1\circ\varphi) &= \sigma\distF^2( (\eta_1)_*\varphi_*\vol,\vol)+\distF^2( (\eta_1)_*\varphi_*\mu_0,\mu_1)\\
		&=\sigma\distF^2(\varphi_*\vol,\eta_1^*\vol)+\distF^2( \varphi_*\mu_0,\eta_1^*\mu_1)= E(\varphi).
\end{align*}
Again, using the invariance of the Fisher--Rao metric we get
\begin{align*}
	E(\varphi\circ\eta_0) &= \sigma\distF^2( \varphi_*(\eta_0)_*\vol, \vol)+\distF^2( \varphi_*(\eta_0)_*\mu_0, \mu_1)
	\\&=\sigma\distF^2( \varphi_*\vol, \vol)+\distF^2( \varphi_*\mu_0, \mu_1)=E(\varphi).
\end{align*}
This proves the lemma.
\end{proof}

Thus, the functional $E$ has two different descending properties:
\begin{enumerate}
	\item It descends to a functional on the right cosets $\Diff_{\vol,\mu_1}(M)\backslash\Diff(M)$, right-invariant with respect to $\Diff_{\vol,\mu_0}(M)$. 
	The corresponding right principal bundle is
	\begin{equation}\label{eq:right_principle_bundle}
	\xymatrix@R3ex@C1ex{
			\Diff_{\vol,\mu_1}(M)\; \ar@{^{(}->}[r] & \Diff(M) \ar[d]
			\\ & \Diff_{\vol,\mu_1}(M)\backslash\Diff(M)\, .
		}
	\end{equation}

	\item It descends to a functional on the left cosets $\Diff(M)/\Diff_{\vol,\mu_0}(M)$, left-invariant with respect to $\Diff_{\vol,\mu_1}(M)$.
	The corresponding left principal bundle is
	\begin{equation}\label{eq:left_principle_bundle}
	\xymatrix@R3ex@C1ex{
			\Diff_{\vol,\mu_0}(M)\; \ar@{^{(}->}[r] & \Diff(M) \ar[d]
			\\ & \Diff(M)/\Diff_{\vol,\mu_0}(M)\, . 
		}
	\end{equation}
\end{enumerate}
We need a strategy to tackle the degeneracy problem explained in \autoref{lem:degeneracy}.
Our approach is simple: we impose the additional constraint on $\varphi$ that it should be connected to the identity by a curve that is $\GI$-orthogonal to the fibres of both principle bundles~\eqref{eq:right_principle_bundle} and~\eqref{eq:left_principle_bundle}.
Then \autoref{prob:inexact_Gdiv} can be solved efficiently by a gradient flow, as we now explain.

\subsection{Gradient flow for inexact matching}

Let $\nablaI E$ denote the gradient of $E$ with respect to the information metric~$\GI$.
Our approach for \autoref{prob:inexact_Gdiv}, i.e., for minimizing the functional $E$ in~\eqref{eq:two_component_matching}, is to discretize the gradient flow
\begin{equation}\label{eq:gradient_flow}
	\dot\varphi = -\nablaI E(\varphi).
\end{equation}
Since the functional $E$ is constant along the fibres of the principal bundles~\eqref{eq:right_principle_bundle} and~\eqref{eq:left_principle_bundle}, the curve traced out by the gradient flow is $\GI$-orthogonal to both fibres, as desired.

Through formulae~\eqref{eq:FR_distance} and \eqref{eq:two_component_matching} we obtain an explicit formula for $E$.
We can then derive the gradient $\nablaI E$. 
It is convenient to carry out the calculations using the sphere representation for the densities via the $W$-map~\eqref{eq:sphere_map}.

\begin{proposition} \label{pro:gradient_of_E}
	The $\GI$--gradient of the divergence-metric matching functional~$E$ in~\eqref{eq:two_component_matching} is given by
	\begin{multline}\label{eq:gradient_of_E}
		\nablaI E(\varphi) = A^{-1}\bigg(
			\sigma c\Big(\int_M\sqrt{\abs{D\varphi^{-1}}}\, \vol\Big)\grad  \sqrt{\abs{D\varphi^{-1}}}  + \\
			 c\Big(\int_M W_\varphi W_1 \vol\Big)\Big(W_1 \grad W_\varphi
			- W_\varphi \grad W_1\Big)
		\bigg)\circ\varphi \, ,
	\end{multline}
	where $A$ is the inertia operator~\eqref{eq:inertia_operator_explicit}, $W_1 \coloneqq W(\mu_1)$, $W_\varphi \coloneqq W(\varphi_*\mu_0)$, and
	\begin{equation}\label{eq:c_fun}
		c\colon [0,1]\to\R, \quad c(x) =  \frac{\arccos\big(\frac{x}{\vol(M)}\big)}{\sqrt{1-\frac{x^2}{\vol(M)^2} }}.
	\end{equation}
\end{proposition}

\begin{proof}
	Take a curve $\varphi(\epsilon)$ such that $\varphi(0) = \varphi$, i.e., a variation of $\varphi$.
	Then
	\begin{multline}\label{eq:first_der}
		\frac{\ud}{\ud \epsilon}\Big|_{\epsilon=0} E(\varphi(\epsilon)) = \GI_{\varphi}(\nabla_G E, \dot\varphi(0)) =
		\\  
		\sigma \left\langle \frac{\delta \distF^{2}}{\delta \mu}(\varphi_*\vol,\vol), \frac{\ud}{\ud \epsilon}\Big|_{\epsilon=0} \varphi(\epsilon)_*\vol \right\rangle +
		\left\langle \frac{\delta \distF^{2}}{\delta \mu}(\varphi_*\mu_0,\mu_1), \frac{\ud}{\ud \epsilon}\Big|_{\epsilon=0} \varphi(\epsilon)_*\mu_0 \right\rangle .
	\end{multline}
	We now write the variation of the form $\dot\varphi(\epsilon) = v\circ\varphi(\epsilon)$ for some vector field $v\in\Xcal(M)$.
	For any $\mu\in\Dens(M)$ we then have
	\begin{equation}\label{eq:Jacobian_variation}
		\frac{\ud}{\ud \epsilon}\Big|_{\epsilon=0} \varphi(\epsilon)_*\mu = -\LieD_{v}\varphi_{*} \mu.
	\end{equation}

	Let $r\coloneqq \sqrt{\mu(M)} = \sqrt{\nu(M)}$.
	Then, from~\eqref{eq:distFF}
	\begin{equation}\label{eq:fisher_distance_flat}
		\distF^{2}(\mu,\nu) = r^2 \arccos\left(\frac{1}{r^2}\int_M W(\mu) W(\nu)\vol \right)^2,
	\end{equation}
	so the variational derivative is given by
	\begin{equation}\label{eq:fisher_distance_flat_derivative}
		\left\langle\frac{\delta \distF^{2}}{\delta \mu}(\mu,\nu),\eta \right\rangle = 
		\underbrace{\frac{\arccos\left(\frac{1}{r^2}\int_M W(\mu) W(\nu)\vol\right)}{\sqrt{1-\left(\frac{1}{r^2}\int_M W(\mu) W(\nu)\vol\right)^2}}}_{c(W(\mu)W(\nu))}
		\left\langle- 2 W(\nu )\frac{\delta W(\mu)}{\delta \mu},\eta\right\rangle
	\end{equation}
	Since $\pair{\frac{\delta W(\mu)}{\delta \mu},\eta} = \frac{1}{2 W(\mu)}\frac{\eta}{\vol}$ we get
	\begin{equation}\label{eq:fisher_distance_flat_derivative2}
		\left\langle\frac{\delta \distF^{2}}{\delta \mu}(\mu,\nu),\eta \right\rangle = -c(W(\mu)W(\nu))\int_M \underbrace{ \frac{W(\nu)}{W(\mu)}}_{F(\mu,\nu)} \eta .
	\end{equation}
	Notice that (i) $F(\mu,\mu)=0$ since the variation $\eta$ preserves the total volume,
	(ii) $\varphi_*F(\mu,\nu) = F(\varphi_*\mu,\varphi_*\nu)$ reflecting the invariance of the Fisher--Rao distance,
	(iii) $\lim_{r\to\infty} c(W(\mu)W(\nu)) = \frac{\pi}{2}$ reflecting the simplified formula when the volume is infinite, and
	(iv) $c(W(\varphi_*\vol)W(\vol)) = c(\sqrt{\abs{D\varphi^{-1}}})$.
	The first term in~\eqref{eq:first_der} now becomes
	\begin{multline}\label{eq:first_term}
		\sigma c(\sqrt{\abs{D\varphi^{-1}}}) \left\langle \frac{\delta \distF^{2}}{\delta \mu}(\varphi_*\vol,\vol), \frac{\ud}{\ud \epsilon}\Big|_{\epsilon=0} \varphi(\epsilon)_*\vol \right\rangle = \\ 
		\sigma c(\sqrt{\abs{D\varphi^{-1}}})\int_M F(\varphi_*\vol,\vol) \LieD_{v}\varphi_*\vol  = \\
		\sigma c(\sqrt{\abs{D\varphi^{-1}}})\int_M F(\varphi_*\vol,\vol) \dd\interior_{v}\varphi_*\vol  = \\
		-\sigma c(\sqrt{\abs{D\varphi^{-1}}})\int_M \dd F(\varphi_*\vol,\vol) \wedge W(\varphi_*\vol)^2\interior_{v}\vol  = \\
		-\sigma c(\sqrt{\abs{D\varphi^{-1}}})\int_M \dd \left( \frac{1}{W(\varphi_*\vol)}\right) \wedge W(\varphi_*\vol)^2\interior_{v}\vol  = \\
		\sigma c(\sqrt{\abs{D\varphi^{-1}}})\int_M \dd W(\varphi_*\vol) \wedge \interior_{v}\vol  = \\
		\sigma c(\sqrt{\abs{D\varphi^{-1}}})\int_M \g(A A^{-1}\grad (W(\varphi_*\vol)), v) \vol  = \\
		\GI_{\id}\left(\sigma c(\sqrt{\abs{D\varphi^{-1}}}) A^{-1}\grad (W(\varphi_*\vol)), \dot\varphi\circ\varphi^{-1} \right)  = \\
		\GI_{\varphi}\left(\sigma c(\sqrt{\abs{D\varphi^{-1}}}) A^{-1}\grad \sqrt{\abs{D\varphi^{-1}}}\circ\varphi, \dot\varphi \right).
	\end{multline}
	Using the notation $W_\varphi \coloneqq W(\varphi_*\mu_0)$ and $W_1\coloneqq W(\mu_1)$, the second term in \eqref{eq:first_der} becomes
	\begin{multline}\label{eq:second_term}
		\left\langle \frac{\delta \distF^{2}}{\delta \mu}(\varphi_*\mu_0,\mu_1), \frac{\ud}{\ud \epsilon}\Big|_{\epsilon=0} \varphi(\epsilon)_*\mu_0 \right\rangle = \\ 
		c(W_\varphi W_1)\int_M F(\varphi_*\mu_0,\mu_1) \LieD_{v}\varphi_*\mu_0  = \\
		c(W_\varphi W_1)\int_M F(\varphi_*\mu_0,\mu_1) \dd\interior_{v}\varphi_*\mu_0  = \\
		-c(W_\varphi W_1)\int_M \dd F(\varphi_*\mu_0,\mu_1) \wedge W_\varphi^2\interior_{v}\vol  = \\
		-c(W_\varphi W_1)\int_M \dd \left( \frac{W_1}{W_\varphi}\right) \wedge W_\varphi^2\interior_{v}\vol  = \\
		c(W_\varphi W_1)\int_M \left( W_1\dd W_\varphi - W_\varphi\dd W_1 \right) \wedge \interior_{v}\vol  = \\
		c(W_\varphi W_1)\int_M \g\left(A A^{-1}( W_1\grad W_\varphi - W_\varphi\grad W_1),v \right) \vol  = \\
		\GI_{\varphi}\left( c(W_\varphi W_1) A^{-1}( W_1\grad W_\varphi - W_\varphi\grad W_1)\circ\varphi, \dot\varphi \right).
	\end{multline}
	Put together, \eqref{eq:first_term} and \eqref{eq:second_term} proves the formula~\eqref{eq:gradient_of_E}.
\end{proof}

Based on \autoref{pro:gradient_of_E}, we can now discretize the gradient flow~\eqref{eq:gradient_flow}.
Indeed, a numerical method is given by the following algorithm.
\begin{algorithm}\label{alg:gradient_flow}
Assume we have a numerical way to represent functions, vector fields, and diffeomorphisms on~$M$, and numerical methods for
(i) composing diffeomorphisms with functions,
(ii) composing diffeomorphisms with diffeomorphisms,
(iii) computing the divergence of a vector field,
(vi) computing the gradient of a functions, and
(v) computing the inverse of the inertia operator~$A$ in~\eqref{eq:inertia_operator_explicit}.
A computational algorithm for the gradient flow~\eqref{eq:gradient_flow} is then given as follows.
\begin{enumerate}
	\item Choose a stepsize $\varepsilon>0$ and initialize the diffeomorphisms $\varphi_0 = \id$, $\varphi^{-1}_{0} = \id$ and a function $J_0 = 1$.
	Precompute $W(\mu_1)$ and $\grad(W(\mu_1))$. Set $k\leftarrow 0$.

	\item Compute action on the source
	\begin{equation}\label{eq:action}
		f_k = W(\mu_0)\circ\varphi^{-1}_k \sqrt{J_k}
	\end{equation}

	\item Compute coefficients $a_k = c\big(\int_M\sqrt{J_k}\, \vol\big)$ and $b_{k} = c\big(\int_M f_k W(\mu_1) \vol\big)$, then compute momentum
	\begin{equation}\label{eq:vel}
		m_k =
			\sigma a_k \grad (\sqrt{J_k})
			+ b_k W(\mu_{1}) \grad f_k
			- b_k f_k \grad(W(\mu_1)).
	\end{equation}
	(If $\vol(M)=\infty$, set $a_k=b_k=1$.)

	\item Compute the vector field infinitesimally generating the negative gradient
	\begin{equation}\label{eq:gradient_comp}
		v_k = -A^{-1}m_k.
	\end{equation}
	
	\item Construct an approximation $\psi^{-1}_k$ to $\exp(-\varepsilon v_k)$, for example $\psi^{-1}_k = \id - \varepsilon v_k$.

	\item Update the inverse diffeomorphism: $\varphi^{-1}_{k+1} = \varphi^{-1}_k\circ\psi^{-1}_k$.

	\item Update the inverse Jacobian using Lie--Trotter splitting\footnotemark
	\begin{equation}\label{eq:trotter_update}
		J_{k+1} = (J_k\circ\psi^{-1}_k)\mathrm{e}^{-\varepsilon\divv(v_k)}.
	\end{equation}
	
	\item Construct an approximation $\psi_k$ to $\exp(\varepsilon v_k)$, for example $\psi_k = \id + \varepsilon v_k$. (output)

	\item Update the forward diffeomorphism: $\varphi_{k+1} = \psi_k\circ\varphi_k$. (output)

	\item Set $k \leftarrow k+1$ and continue from step 2.

\end{enumerate}
\end{algorithm}
\footnotetext{Here one could use the Strang splitting to obtain higher order.}

\subsection{Geometry of the gradient flow} \label{sub:densdensaction}
In this section we describe the geometry associated with the divergence-metric functional $E$ in equation~\eqref{eq:two_component_matching} and the corresponding gradient flow~\eqref{eq:gradient_flow}.

The diffeomorphism group acts diagonally on $\Dens(M)\times\Densbar(M)$ by $\varphi\cdot(\mu,\nu) = (\varphi_*\mu,\varphi_*\nu)$.
The isotropy group of $(\mu,\nu) \in \Dens(M)\times\Densbar(M)$, i.e., the subgroup of $\Diff(M)$ that leaves $(\mu,\nu)$ invariant, is given by 
\begin{equation}\label{eq:two_component_isotropy_group}
	\Diff_{\mu,\nu}(M)\coloneqq\Diff_{\mu}(M)\cap\Diff_{\nu}(M).	
\end{equation}
The action of $\Diff(M)$ on $\Dens(M)\times\Densbar(M)$ is \emph{not} transitive, so there is more than one group orbit.
The group orbit through $(\vol,\mu_0)$, given by 
\begin{equation}\label{eq:two_component_orbit}
	\mathrm{Orb}(\vol,\mu_0) \coloneqq \Diff(M)\cdot(\vol,\mu_0)\subset \Dens(M)\times\Dens(M),
\end{equation}
is a way to represent the quotient set $\Diff(M)/\Diff_{\vol,\mu_0}(M)$.
This set is potentially complicated (an orbifold), but let us assume we stay away from singular points so we can work with $\mathrm{Orb}(\vol,\mu_0)$ as a submanifold of $\Dens(M)\times\Densbar(M)$.
The principal bundle~\eqref{eq:left_principle_bundle} can then be represented as
\begin{equation}\label{eq:new_principle_bundle}
\xymatrix@R3ex@C3ex{
			\Diff_{\vol,\mu_0}(M)\; \ar@{^{(}->}[r] & \Diff(M) \ar[d]^-{\pi_{\vol,\mu_0}} 
			\\ & \mathrm{Orb}(\vol,\mu_0)\, ,
		}
\end{equation}
where
\begin{equation}\label{eq:big_projection}
	\pi_{\vol,\mu_0}(\varphi) = \varphi\cdot(\vol,\mu_0) = (\varphi_*\vol,\varphi_*\mu_0).
\end{equation}

The manifold $\Dens(M)\times\Densbar(M)$ comes with a metric, namely
\begin{equation}\label{eq:GFF_metric}
	\GFF_{(\mu,\nu)}\big((\alpha_\mu,\alpha_\nu),(\beta_\mu,\beta_\nu)\big) = \sigma \GF_{\mu}(\alpha_{\mu},\beta_{\mu}) + \GF_{\nu}(\alpha_{\nu},\beta_{\nu}).
\end{equation}
(Notice that $\GF$ is naturally extended to a metric on $\Densbar(M)$ by the $W$-map~\eqref{eq:sphere_map}.)
The corresponding distance is given by
\begin{equation}\label{eq:distFF}
	\distFF^2((\mu_1,\nu_1),(\mu_2,\nu_2)) = \sigma \distF^{2}(\mu_1,\mu_2) + \distF^{2}(\nu_1,\nu_2).
\end{equation}
To connect back to the divergence-metric matching in \autoref{prob:inexact_Gdiv}, the key observation is that 
\begin{equation}
	E(\varphi) = \distFF^{2}((\vol,\mu_1),(\vol,\mu_0)\cdot\varphi) = \distFF^{2}((\vol,\mu_1),\pi_{\vol,\mu_0}(\varphi)),	
\end{equation}
so $E(\varphi)$ is simply the function $(\mu,\nu)\mapsto \distFF^{2}((\vol,\mu_1),(\mu,\nu))$ on $\Dens(M)\times\Densbar(M)$ lifted to $\Diff(M)$.

Let us now discuss how the metric $\GI$ fits into this geometry.

\begin{lemma}\label{lem:two_component_descend}
	The information metric $\GI$ on $\Diff(M)$ is descending with respect to the principal bundle~\eqref{eq:new_principle_bundle}, i.e., it descends to a metric $\Gorb$ on $\mathrm{Orb}(\vol,\mu_0)$.
\end{lemma}

\begin{proof}
	Translation by $\psi\in\Diff_{\vol,\mu_0}(M)$ of a vector $U\in T_{\varphi}\Diff(M)$ along the fibre of~\eqref{eq:new_principle_bundle} is given by $U\mapsto U\circ\psi$.
	Therefore, the metric $\GI$ is descending if and only if
	\begin{equation}\label{eq:descending_condition}
		\GI_{\varphi}(U,V) = \GI_{\varphi\circ\psi}(U\circ\psi,V\circ\psi), \quad \forall\, U,V\in \mathcal{H}^{\vol,\mu_0}_{\varphi},
	\end{equation}
	where $\mathcal{H}^{\vol,\mu_0}_{\varphi}$ is the distribution that is $\GI$-orthogonal to the tangent spaces of the fibres of~\eqref{eq:new_principle_bundle}.
	Since $\GI$ is right-invariant, condition~\eqref{eq:descending_condition} is automatically true.
\end{proof}

Although $\GI$ descends to a metric $\Gorb$, the associated distance function $\distorb$ is \emph{not} explicitly computable, in particular it is not given by $\distFF$.
If it was, the result in \autoref{lem:two_component_descend} would allow us to use the same technique as in \autoref{sec:density_matching} to solve the non-compatible, non-degenerate matching problem using $\GI$ and $\GF$, namely to lift the geodesics on $\mathrm{Orb}(\vol,\mu_0)$.
Our remedy is to exchange $\distorb$ for the distance function $\distFF$ on the ambient manifold $\Dens(M)\times\Densbar(M)$.
The geometry of our set-up is illustrated in \autoref{fig:two_component}.

\begin{figure}
	\centering
	\includegraphics[width=0.7\textwidth]{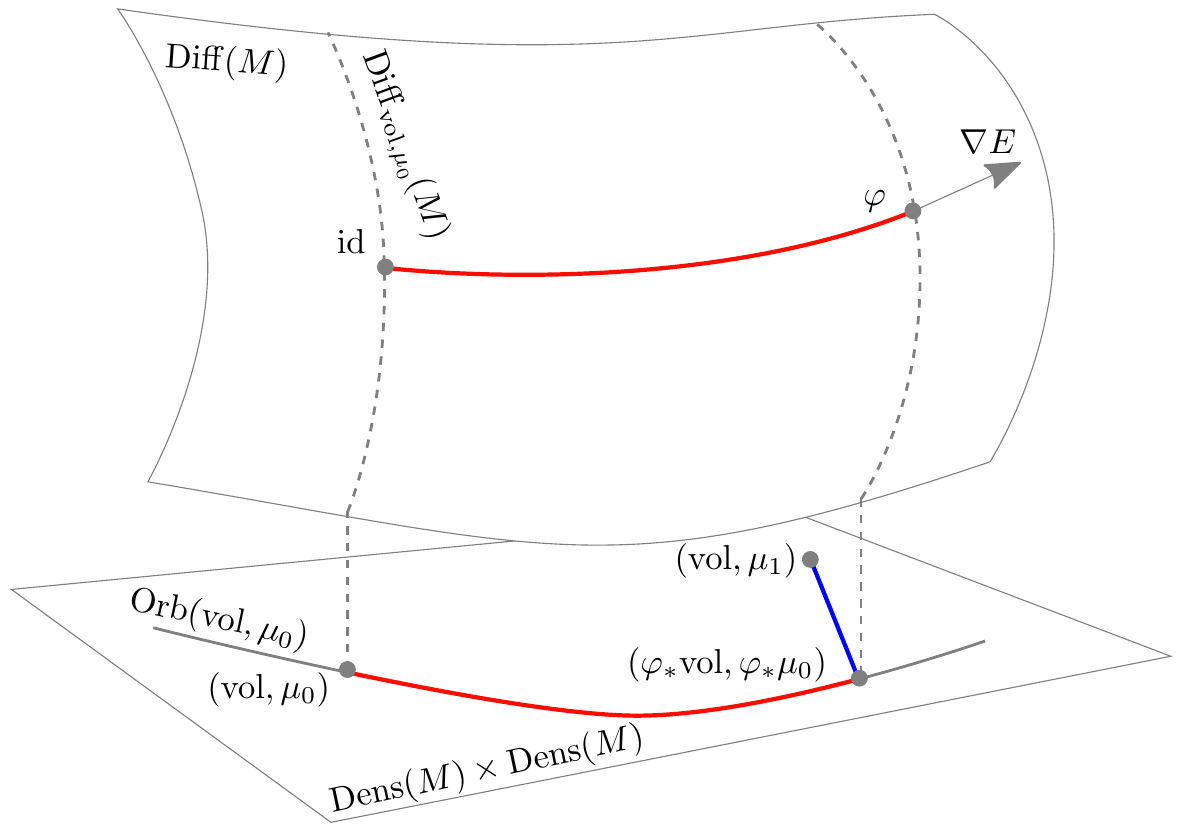}
	\caption{Illustration of the geometry associated with inexact density matching using the divergence-metric.
	The gradient flow on $\Diff(M)$ descends to a gradient flow on the orbit $\mathrm{Orb}(\vol,\mu_0)$.
	While constrained to $\mathrm{Orb}(\vol,\mu_0)\subset \Dens(M)\times\Densbar(M)$, this flow strives to minimize the ambient Fisher--Rao distance~$\distFF$ to $(\vol,\mu_1)$.
	}
	\label{fig:two_component}
\end{figure}

\subsection{Two-component gradient flow}

Since both the information metric $\GI$ and the functional $E$ descend with respect to~\eqref{eq:new_principle_bundle}, the gradient flow~\eqref{eq:gradient_flow} induces a gradient flow on $\mathrm{Orb}(\vol,\mu_0)\subset \Dens(M)\times\Densbar(M)$.
This allows us to represent the gradient flow on~$\Dens(M)\times\Densbar(M)$.

\begin{proposition}\label{lem:descending_gradient_flow}
	The gradient flow~\eqref{eq:gradient_flow} descends to a two-component gradient flow equation on $\Dens(M)\times\Densbar(M)$, constrained to stay on $\mathrm{Orb}(\vol,\mu_0)$.
	Expressed in the variables $J = \abs{D\varphi^{-1}}$ and $P = (\varphi_*\mu_0)/\vol$ it is given by
	\begin{equation}\label{eq:constrained_gradient_flow_explicit}
		\begin{split}
			\dot J &= -v\cdot \grad J - J \divv(v) \\
			\dot P &= -v\cdot \grad P - P \divv(v)
		\end{split}
	\end{equation}
	with
	\begin{multline}
		v = A^{-1}\bigg(
			\sigma c\Big(\int_M\sqrt{J}\, \vol\Big)\grad  \sqrt{J}  + \\
			 c\Big(\int_M W(\mu_1) \sqrt{P}\, \vol\Big)\Big(W(\mu_1) \grad \sqrt{P}
			- \sqrt{P} \grad W(\mu_1)\Big)
		\bigg).
	\end{multline}
\end{proposition}

\begin{proof}
	Let $\varphi(t)$ be the solution of the gradient flow~\eqref{eq:gradient_flow}.
	Then $\dot\varphi(t)\circ\varphi(t)^{-1} = \nablaI E(\varphi(t))\circ\varphi(t)^{-1} \eqqcolon v(t)$ depends on $\varphi(t)$ only through $\varphi(t)_*\vol$ and $\varphi(t)_*\mu_0$, i.e., through $J$ and $P$.
	We also have
	\begin{equation}
		\frac{\dd}{\dd t} \varphi(t)_*\vol = -\LieD_{v(t)}\varphi(t)_*\vol = -\left(\LieD_{v(t)}J + \divv(v(t))J\right)\vol
	\end{equation}
	and
	\begin{equation}
		\frac{\dd}{\dd t} \varphi(t)_*\mu_0 = -\LieD_{v(t)}\varphi(t)_*\mu_0 = -\left(\LieD_{v(t)}P + \divv(v(t))P\right)\vol.
	\end{equation}
	This proves the result.
\end{proof}

This proposition tells us that an alternative way to compute the gradient flow~\eqref{eq:gradient_flow} is to solve first~\eqref{eq:constrained_gradient_flow_explicit} and then the lifting equations for the principle bundle~\eqref{eq:new_principle_bundle}, thereby obtaining a horizontal curve in~$\Diff(M)$.
We have not investigated this approach in detail.

\subsection{Relation to matching with compatible background metric}

In this section we want to show the relation between the compatible background approach in \autoref{sec:density_matching} and the gradient flow approach developed here.
Recall that the solutions to \autoref{prob:oit} and \autoref{prob:inexact_Gdiv} are obtained by the inverse of the endpoint of a horizontal geodesic, obtained by lifting a Fisher--Rao geodesic.
As a consequence, we have the following two results.

\begin{lemma}\label{thm:relation_to_oit}
	Let $\mu_0 =\vol, \mu_1\in\Dens(M)$ and let $[0,1]\mapsto \gamma(t)\in\Diff(M)$ be the horizontal $\GI$-geodesic such that $\gamma(1)^{-1}$ solves the OIT problem (\autoref{prob:oit}).
	Then $\nabla E(\gamma(t);\mu_1,\vol)$ is parallel with $\dot\gamma(t)$.
\end{lemma}

\begin{proof}
	We have $E(\varphi;\mu_1,\vol) = \sigma\distF^{2}(\varphi^*\vol,\vol) + \distF^{2}(\varphi^*\vol,\mu_1)$.
	Thus, $E(\cdot;\mu_1,\vol)$ descends with respect to Moser's principal bundle~\eqref{eq:principal_bundle_diff_dens}.
	Since the information metric $\GI$ descends to the Fisher--Rao metric, the gradient flow descends to a gradient flow on $\Dens(M)$, given by
	\begin{equation}\label{eq:descending_gradient_flow_moser}
		\dot\mu = -\nablaF e(\mu), \quad e(\mu) = \sigma\distF^{2}(\vol,\mu) + \distF^{2}(\mu,\mu_1),
	\end{equation}
	where $\nablaF$ is the gradient with respect to~$\GF$.
	If now $\mu = \pi(\gamma(t))$, then $\nablaF e(\mu)$ is parallel with $\frac{\dd}{\dd t}\pi(\gamma(t))$, since $\pi(\gamma(t))$ is the unique minimizing geodesic between $\vol$ and $\mu_1$.
	The result now follows since two horizontally lifted paths are parallel if they are parallel on~$\Dens(M)$.
\end{proof}

\begin{proposition}\label{pro:gradient_flow_and_oit}
	If $\mu_0=\vol$ and $\mu_1\in\Dens(M)$, then the limit of the gradient flow $\dot\varphi = -\nabla_A E(\varphi;\mu_1,\vol)$, $\varphi(0)=\id$ exists and coincides with the inverse of the solution to the non-exact OIT problem (\autoref{prob:oit_inexact}).
\end{proposition}
\begin{proof}
	Follows since the gradient flow $\dot\varphi = -\nabla_A E(\varphi;\mu_1,\vol)$ descends to the gradient flow~\eqref{eq:descending_gradient_flow_moser} and $e(\mu)$ is a strictly convex functional and $\Dens(M)$ is a convex space with respect to the Fisher--Rao geometry.
\end{proof}

\begin{remark}
In contrast to the previous parts of this section we actually require here that $\mu_1$ is strictly positive. This is indeed a necessary condition: for $\mu_1$ 
on the boundary of $\on{Dens}(M)$ we cannot guarantee the existence of minimizer to the non-exact optimal information transport problem. 
In fact for a target density $\mu_1\in \Densbar(M)$ the optimal deformation $\varphi$ will in general not be a diffeomorphisms, instead it will have a vanishing derivative on certain points or even intervals. 
To guarantee the existence of minimizers also in this situation, one could use a complete metric on the diffeomorphism group as regularization term. 
Possible choices for this include higher order Sobolev metrics \cite{EbMa1970,BEK2014,BrVi2014,MiMu2013} and metrics that are induced by Gaussian kernels \cite{TrYo2005b}. 
\end{remark}

\section{Examples of matching with non-compatible metric}\label{sec:density_matching_non_compatible_examples}

In this section we evaluate the gradient flow based \autoref{alg:gradient_flow} in various examples where $\mu_0,\mu_1\in\Densbar(\mathbb{T}^2)$, i.e., there are regions with vanishing density (represented by black pixels).

For simplicity, in our experiments we consider the flat, periodic torus and use fast fouler transform for inverting the operator~$A$ in~\eqref{eq:inertia_operator_explicit}.
There has been extensive work on fast, efficient solution of Poisson's problem on other manifolds.
See, for example, the review~\cite{DzEl2013}.

\subsection{Registration of letters: J to V}

This example illustrates the capability to produce severely deformed warps, namely to warp the letter~J into the letter~V.
We also examine the effect of the balancing parameter~$\sigma$.

With step-size $\varepsilon = 0.2$, balance $\sigma=0.05$, and $400$~iterations, we get the following energy evolution and sequence of warps.

\begin{center}
\includegraphics[width=.89\textwidth]{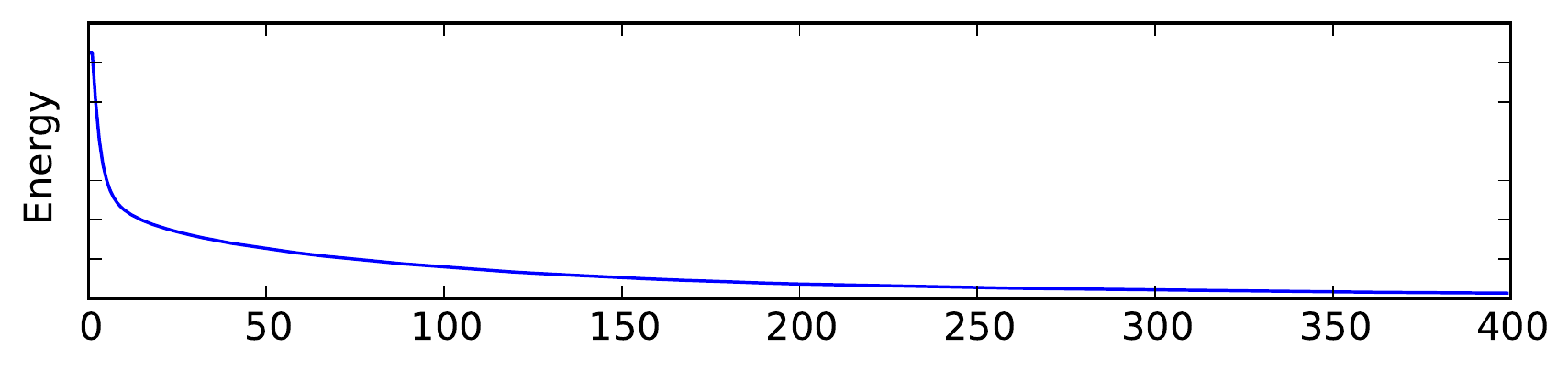}
\\
\begin{tabular}{@{}c@{}c@{}c@{}c@{}c@{}c@{}}
	\includegraphics[width=.16\textwidth]{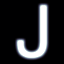}\,
	&	
	\includegraphics[width=.16\textwidth]{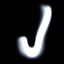}\,
	&	
	\includegraphics[width=.16\textwidth]{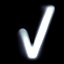}\,
	&	
	\includegraphics[width=.16\textwidth]{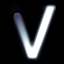}\,
	&	
	\includegraphics[width=.16\textwidth]{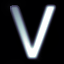}\,
	&	
	\includegraphics[width=.16\textwidth]{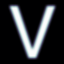}
	\\[-1ex]
	\small{Source} & & & & \small{Final} & \small{Target}
\end{tabular}
\end{center}

\noindent 
Notice in the warp that the pixel-values has changed in regions of large deformations (the upper left part of the final~V).
This is due to expansion.
The corresponding mesh deformation and Jacobian of the inverse at the final point look as follows.
\begin{center}
\begin{tabular}{@{}cc@{}}
	\small Inverse warp & \small Inverse Jacobian $\abs{D\varphi^{-1}}$ \\
	\includegraphics[height=.4\textwidth]{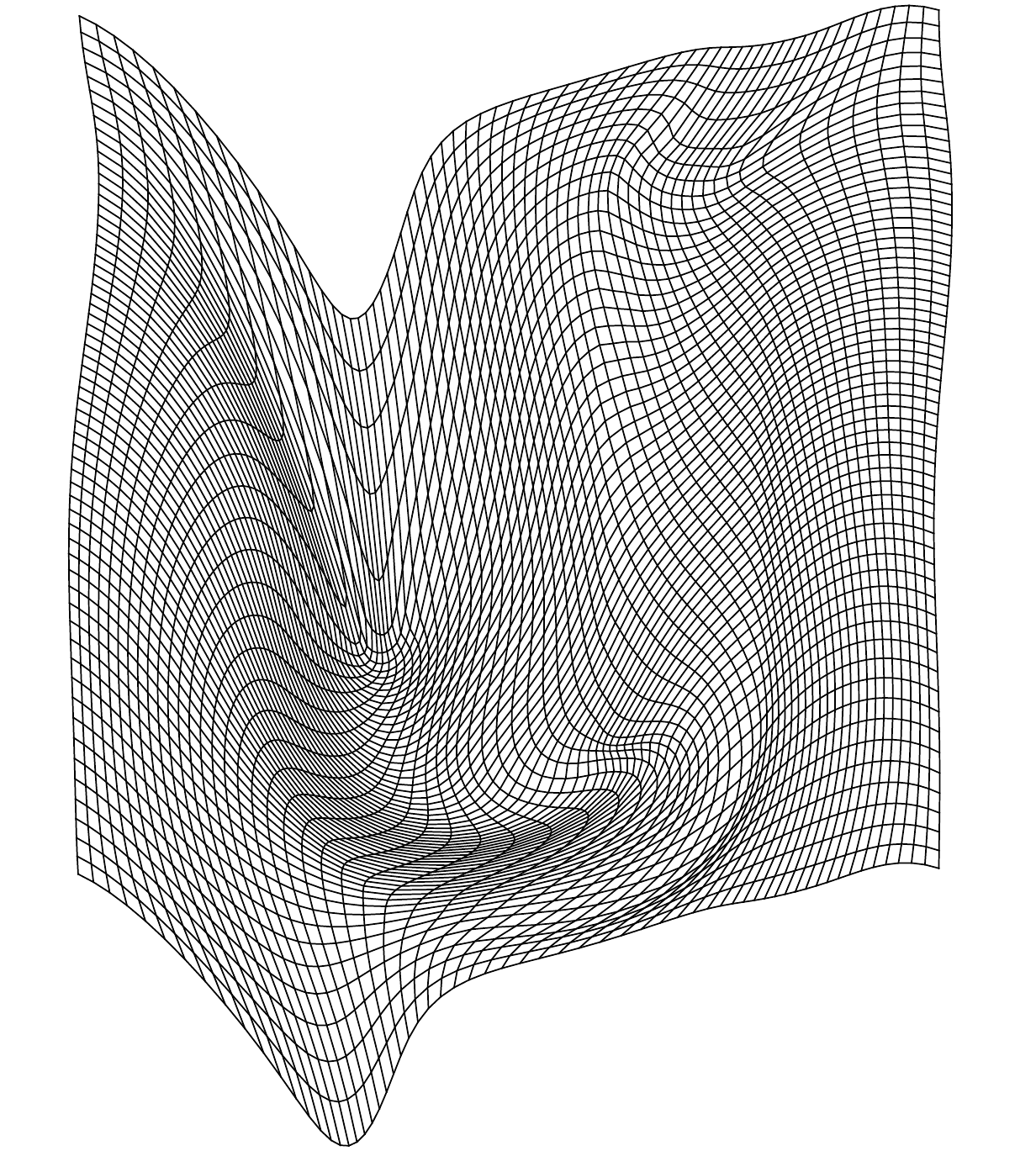}\,
	&	
	\includegraphics[height=.4\textwidth]{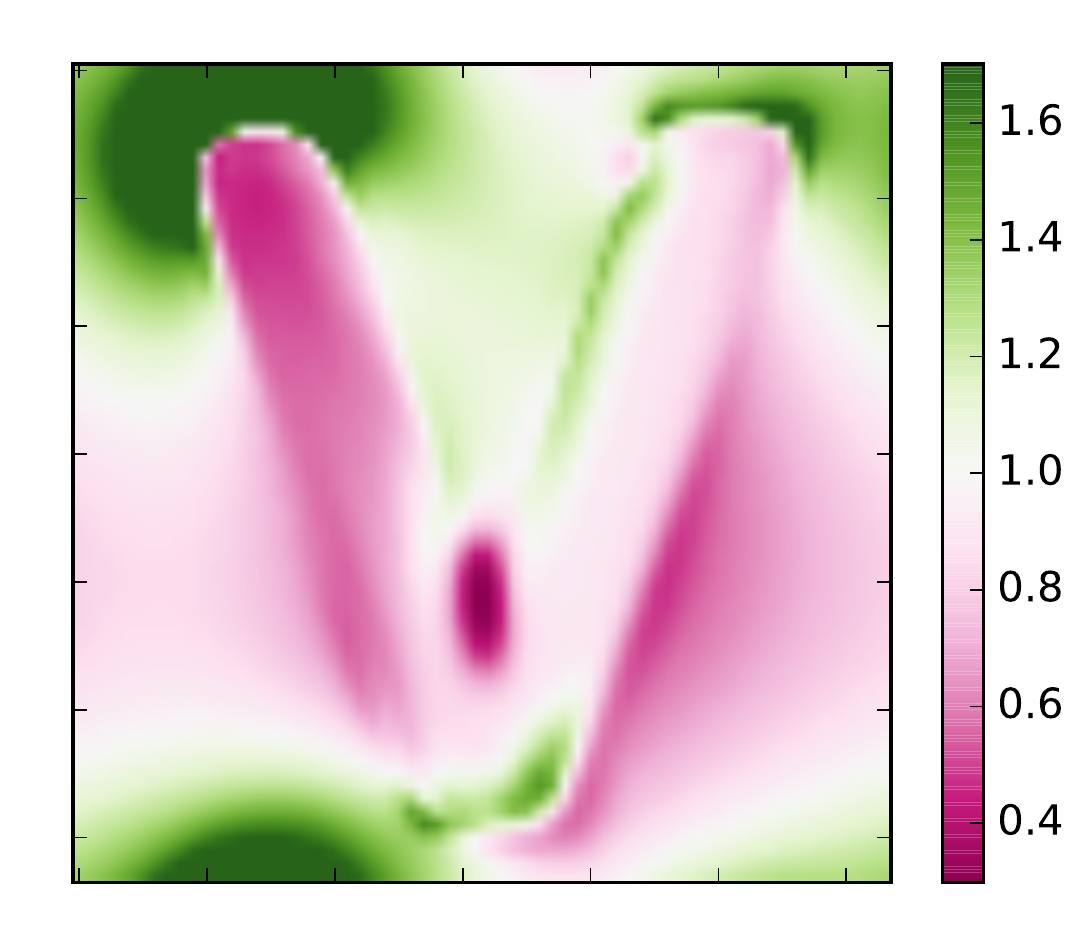}
\end{tabular}
\end{center}
\noindent
Notice how the lower part of the~J is stretched out to the left part of the~V.

Let us now do the same experiment but with the larger balancing parameter $\sigma = 5$.

\begin{center}
\includegraphics[width=.89\textwidth]{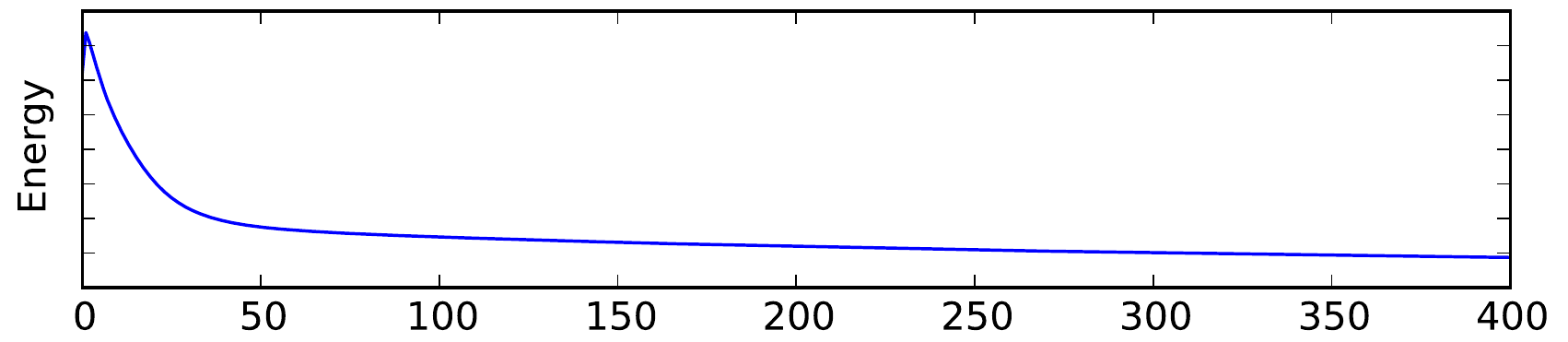}
\\
\begin{tabular}{@{}c@{}c@{}c@{}c@{}c@{}c@{}}
	\includegraphics[width=.16\textwidth]{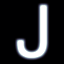}\,
	&	
	\includegraphics[width=.16\textwidth]{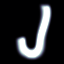}\,
	&	
	\includegraphics[width=.16\textwidth]{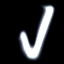}\,
	&	
	\includegraphics[width=.16\textwidth]{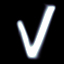}\,
	&	
	\includegraphics[width=.16\textwidth]{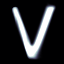}\,
	&	
	\includegraphics[width=.16\textwidth]{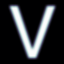}
	\\[-1ex]
	\small{Source} & & & & \small{Final} & \small{Target}
\end{tabular}
\end{center}

\noindent 
Notice here that the warp rarely changes the pixel-values, at the price of less expansion of the left part of the~V.
This is due a smaller amount of compression and expansion, as seen below in the corresponding mesh deformation and Jacobian of the inverse.
\begin{center}
\begin{tabular}{@{}cc@{}}
	\small Inverse warp & \small Inverse Jacobian $\abs{D\varphi^{-1}}$ \\
	\includegraphics[height=.4\textwidth]{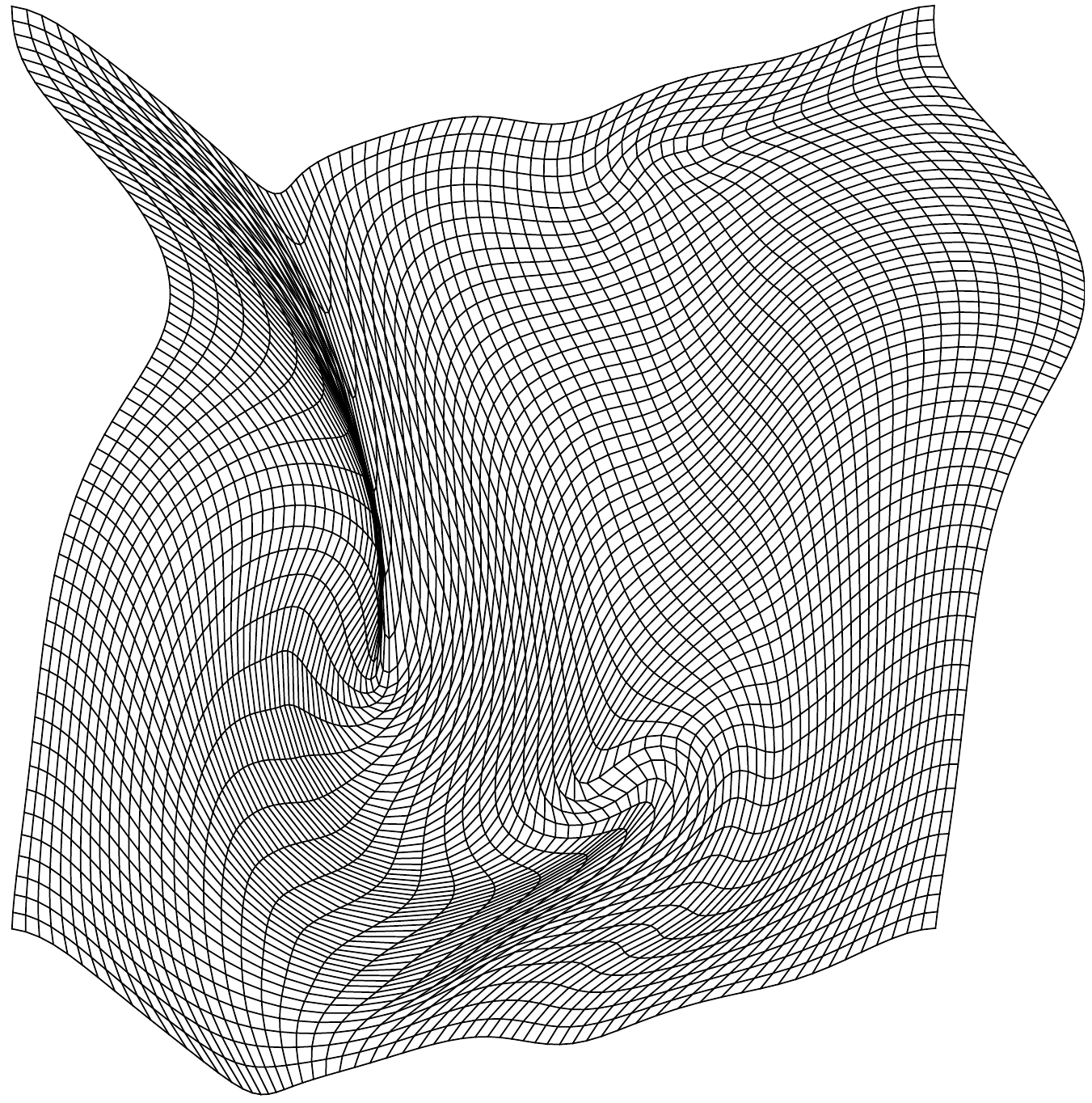}\,
	&	
	\includegraphics[height=.4\textwidth]{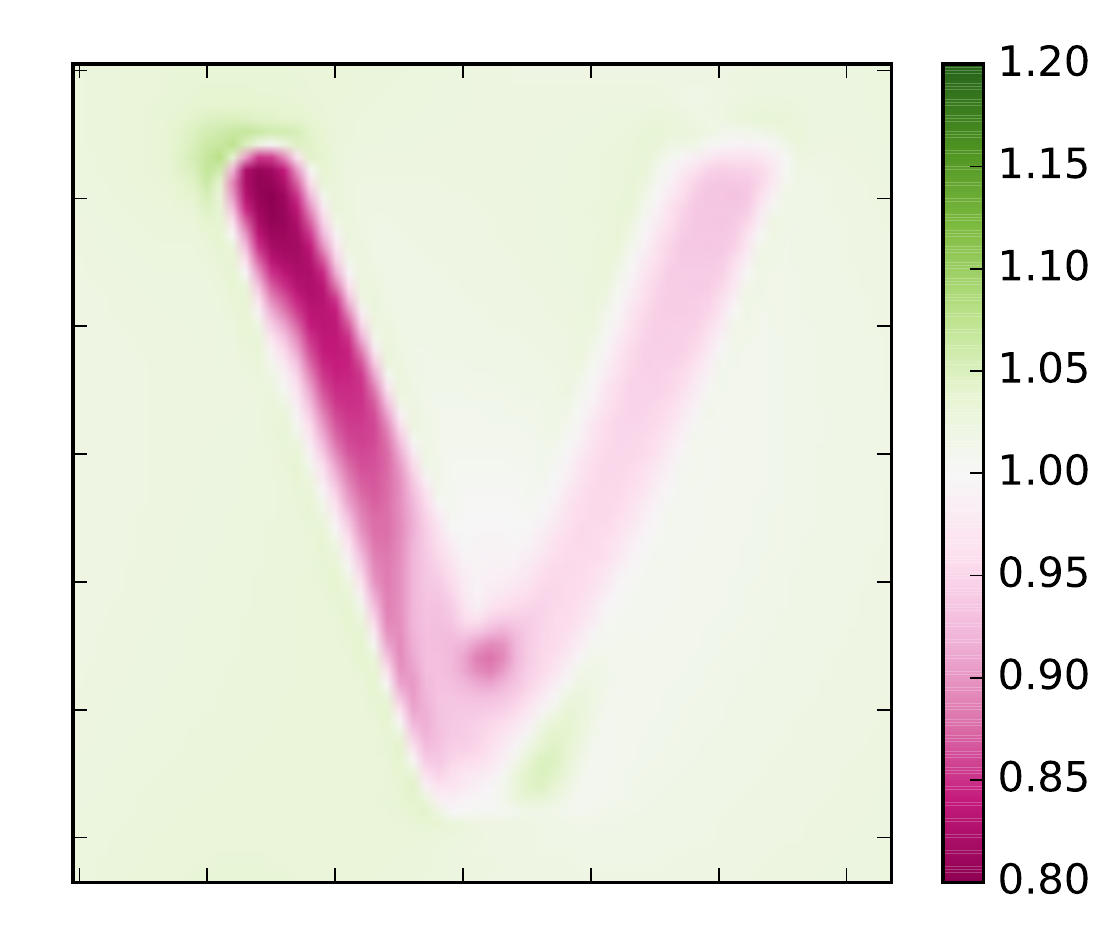}
\end{tabular}
\end{center}
\noindent
Compared to the smaller $\sigma$, the Jacobian determinant is more regular and closer to one and the mesh deformation is almost volume-preserving.
This example illustrates nicely the influence of the balancing parameter.

\subsection{Registration of noisy x-ray images}

In this example we illustrate the use of \autoref{alg:gradient_flow} for registration of low-resolution, noisy x-ray images of human hands.
With step-size $\varepsilon = 0.2$, balance $\sigma=0.1$, and $400$~iterations, we get the following energy evolution and sequence of warps.

\begin{center}
\includegraphics[width=.89\textwidth]{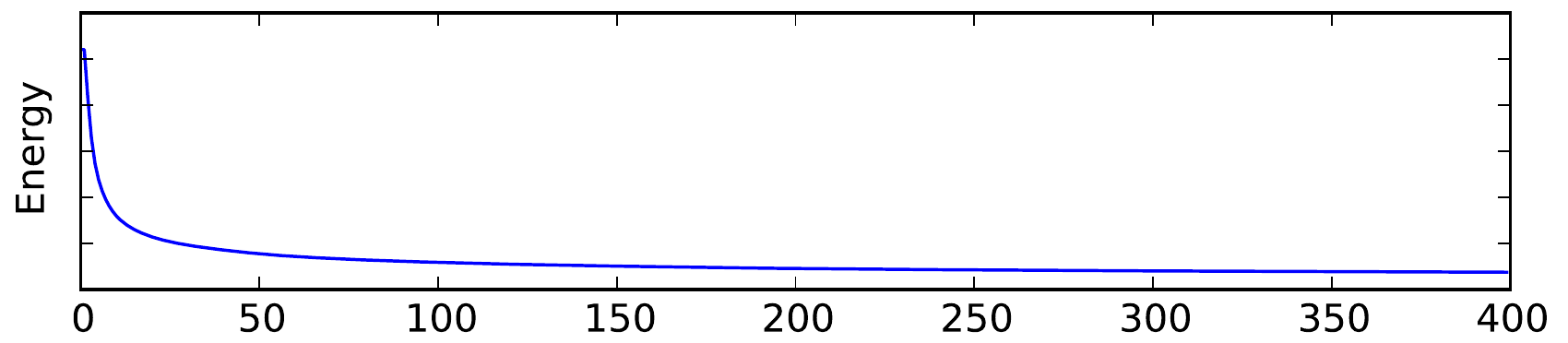}
\\
\begin{tabular}{@{}c@{}c@{}c@{}c@{}c@{}c@{}}
	\includegraphics[width=.16\textwidth]{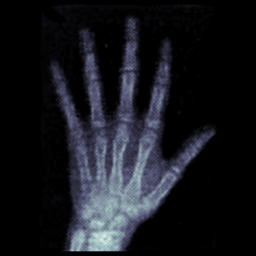}\,
	&	
	\includegraphics[width=.16\textwidth]{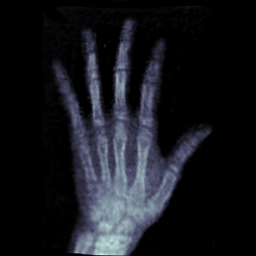}\,
	&	
	\includegraphics[width=.16\textwidth]{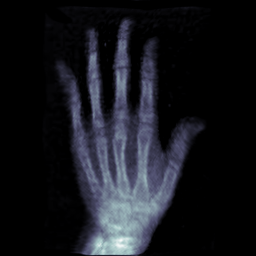}\,
	&	
	\includegraphics[width=.16\textwidth]{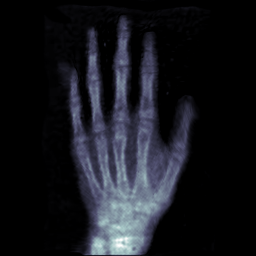}\,
	&	
	\includegraphics[width=.16\textwidth]{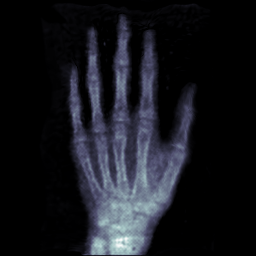}\,
	&	
	\includegraphics[width=.16\textwidth]{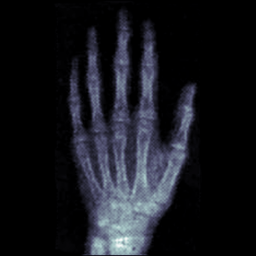}
	\\[-1ex]
	\small{Source} & & & & \small{Final} & \small{Target}
\end{tabular}
\end{center}

\noindent 
Except for the tip of the little finger and the thumb, the resulting path of diffeomorphisms yields a good warp between corresponding bone structures in the hands.
The mesh deformation and Jacobian of the inverse at the final point look as follows.
\begin{center}
\begin{tabular}{@{}cc@{}}
	\small Inverse warp & \small Inverse Jacobian $\abs{D\varphi^{-1}}$ \\
	\includegraphics[height=.4\textwidth]{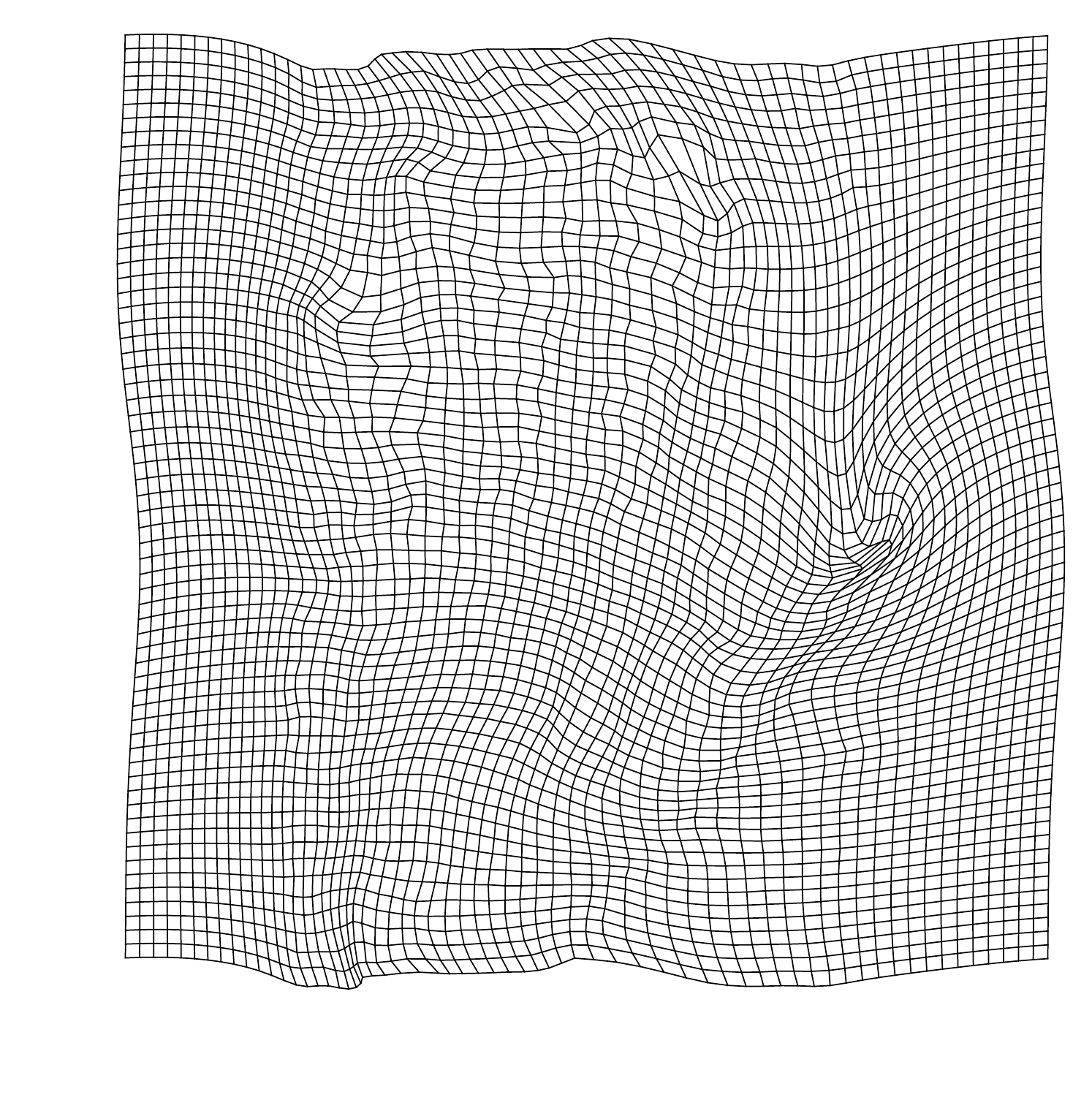}\,
	&	
	\includegraphics[height=.4\textwidth]{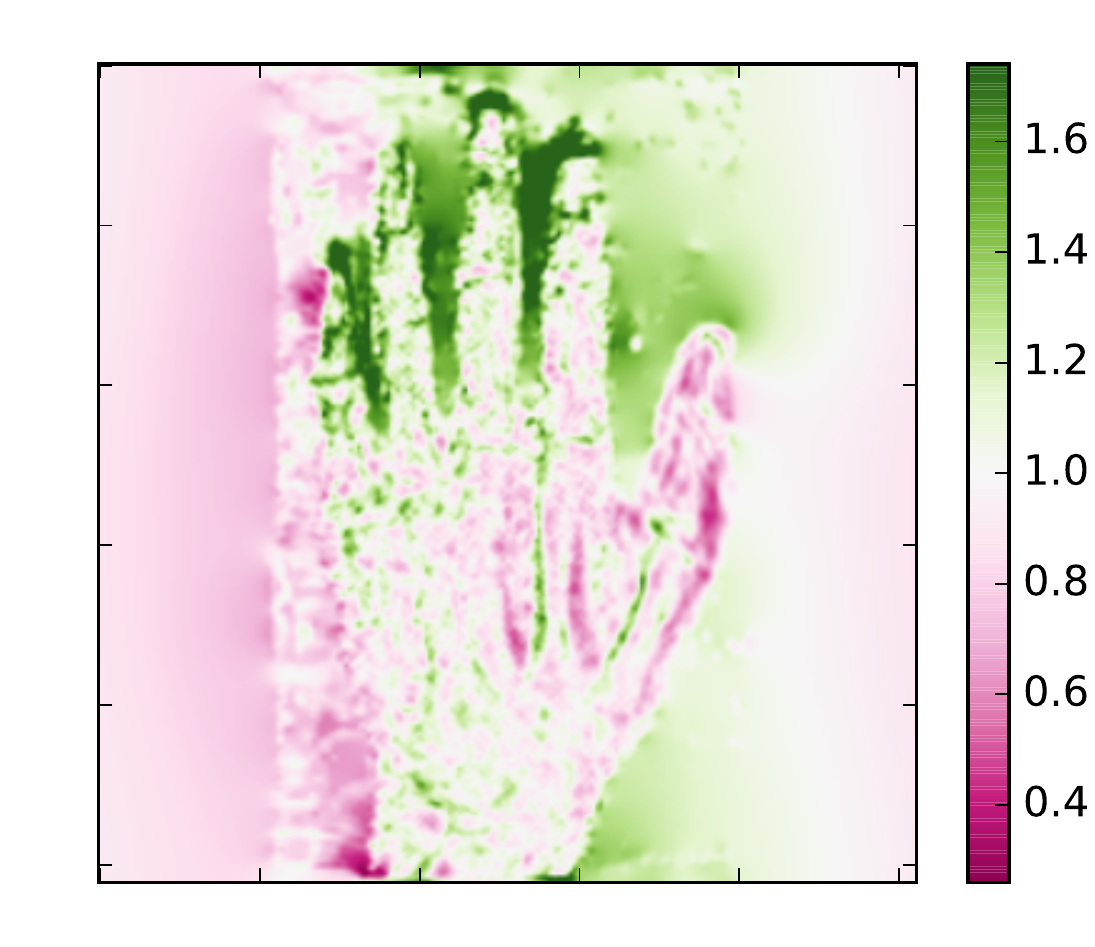}
\end{tabular}
\end{center}
\noindent
Here one can see how the noise affect the diffeomorphisms: both the mesh warp and the Jacobian are somewhat irregular, except at the left and right borders where the source and target densities vanish.

\appendix

\section{Constructing compatible background metrics} \label{sec:choice_of_background_metric}
In \autoref{sec:density_matching} we have described a method to solve the density matching problem, assuming that we are given a compatible background  metric. If the source density $\mu_0$ is not equal to the density induced by the background metric, one has to construct such a metric first (as in~\autoref{ex:compatible:imaging}).  
There is, of course, a range of background metrics $\mathsf{h}$ having a prescribed volume form; here we describe two specific choices.

\subsection{Conformal metric}\label{conformalmetric}
As already discussed, any volume form $\mu\in \Dens(M)$ can be written
$\mu=I\vol_{\g}$.
Note that $I$ is the Radon--Nikodym derivative of $\mu$ with respect to $\vol_{\g}$ (as measures). 
This observation yields a first choice for the desired metric.
\begin{lemma}
Let $\mu=I\vol_{\g} \in \Dens(M)$. 
Then the modified metric 
\begin{equation}\label{eq:conformal}
	\mathsf{h} = I^{\frac{2}n} \g
\end{equation}
is compatible with $\mu$, i.e., $\vol_{\mathsf{h}} = \mu$.
\end{lemma}

\begin{proof}
	In coordinates $(x^1,\ldots,x^n)$ we have 
	\begin{equation*}
	 \vol_{\mathsf{h}}=\sqrt{\det (\mathsf{h}_{ij})}\, dx^1\wedge\ldots\wedge dx^n\,.\qedhere
	\end{equation*}
\end{proof}

\begin{remark}
	Note that the metric~\eqref{eq:conformal} is conformally equivalent to the background metric $\g$. 
	In mechanics, it is called the \emph{Jacobi metric}.
\end{remark}

The  advantage of the metric~\eqref{eq:conformal} is that it is easy to construct and also that the Laplacian and gradient take simple forms (in terms of the Laplacian and gradient of the original background metric $\g$.)
\begin{lemma}
	The Laplacian and gradient of the metric $\mathsf{h}=I^\frac{n}{2}\g$ are given by
	\begin{equation}
		\grad_{\mathsf{h}}(f)= I^{-\frac{n}{2}} \grad_{\g}(f),\qquad \Delta_{\mathsf{h}} f=\frac1I \divv_{\g}(I^{1-\frac{n}{2}} \grad_{\g}(f))
	\end{equation}
\end{lemma}

\begin{proof}
	To calculate the expression for the gradient we use
	\begin{align*}
		df(X)= \mathsf{h}(\grad_{\mathsf{h}}(f),X) = \g(I^\frac{n}{2}\grad_{\mathsf{h}}(f),X)=\g(\grad_{\g}(f),X)
	\end{align*}
	For the Laplacian we express the divergence in coordinates
	\begin{align*}
		\divv_g(X)=\frac{1}{\sqrt{\det \mathsf{h}_{ij}}}\partial_i(\sqrt{\det \mathsf{h}_{ij}}\,X^i)=\frac1I \partial_i(I\,X^i)=\frac1I \divv_{\g}(IX)\,.
	\end{align*}
	Then we have
	\begin{align*}
		\Delta_{\mathsf{h}} f &=  \divv_{\mathsf{h}}(\grad_{\mathsf{h}}(f))=\divv_{\mathsf{h}} (I^{-\frac{n}{2}} \grad_{\g}(f))=\frac1I \divv_{\g}(I \cdot I^{-\frac{n}{2}} \grad_{\g}(f))\\&=\frac1I \divv_{\g}(I^{1-\frac{n}{2}} \grad_{\g}(f))\qedhere
	\end{align*}
\end{proof}

\begin{remark}
	 Note that for $n=2$ the formula for the Laplacian simplifies to
	$$\Delta_{\mathsf{h}} =\frac1I \Delta_{\g}\,,$$ reflecting the scale invariance of the Laplacian in  dimension two.
\end{remark}

\begin{corollary}
Using the metric~\eqref{eq:conformal} the lifting equation \eqref{eq:veq} reads
\begin{equation}
	\begin{split}
		\dot\ph(t) &= v(t)\circ\ph(t), \quad \ph(0) = \on{id}, \\
		v(t) &= I^{-\frac{n}{2}} \grad_{\g}(f), \\
		\frac1I \divv_{\g}(I^{1-\frac{n}{2}} \grad_{\g}(f)) &= \frac{\dot\mu(t)}{\mu(t)}\circ \ph(t)^{-1}\, .
	\end{split}
\end{equation}
\end{corollary}

\subsection{Constructing a flat compatible metric}\label{flatmetric}
For a flat background metric $\g$, the Jacobi metric $\mathsf{h}$ given by~\eqref{eq:conformal} is not, in general, flat.
Indeed, $\mathsf{h}$ is only flat if~$I$ is constant. 
Now we will describe a method to choose a flat background metric assuming that the background metric $\g$ is flat.

\begin{lemma}
	Assume that $M$ carries a flat background metric $\g$. 
	Let $\mu\in \Dens(M)$ and let $\varphi$ be the solution of \autoref{prob:oit} with $\mu_0 = \vol$ and $\mu_1 = \mu$.
	Then $\mathsf{h}=\varphi^*\g$ is a flat metric and $\vol_{\mathsf{h}} = \mu$.
\end{lemma}

\begin{proof}
	We have $\vol_{\ph^*\mathsf{h}}=\ph^*\vol_{\mathsf{h}}$ for any metric $\mathsf{h}$ and therefore $\mathsf{h}=\ph^*\g$
	has the desired volume density. 
	The flatness follows since the curvature tensor $\mathcal R_{\g}$ of $\g$ satisfies $0= \mathcal R_{\g}=\ph_{*} \mathcal R_{\ph^* \g}$.
\end{proof}

\subsection{Symmetric matching}
In the previous section we have described how to choose a metric that is compatible with a fixed volume form. 
Thus, the solution of \autoref{prob:exact} with respect to the Fisher--Rao metric \eqref{eq:fisher_rao_metric} can be obtained as follows:
\begin{enumerate}
 \item Given $\mu_0$, $\mu_1 \in\Dens(M)$, construct a background metric $\mathsf{h}$ compatible with $\mu_0$ by \autoref{conformalmetric} or \autoref{flatmetric}.
 \item Solve the lifting equation \eqref{eq:veq} with respect to~$\mathsf{h}$, using the explicit geodesic curve~\eqref{eq:fisher_rao_geodesics} between $\mu_0$ and $\mu_1$.
\end{enumerate}
This problem is not symmetric in $\mu_0$ and $\mu_1$, since the metric $\mathsf{h}$ depends on~$\mu_0$.
However, the condition that $\mathsf{h}$ is compatible with $\mu_{0}$ can be made more general, which will allow us to derive a symmetric solution.
We use the fact that any geodesic that is horizontal at some~$s\in (0,1)$ remains horizontal for all time. 
Thus we do not have to choose a metric that is compatible with one of the prescribed densities $\mu_0$ and $\mu_1$, but only a metric 
that is compatible with some density that lies on the geodesic connecting the two. 
This suggests the following choice of  background metric $\mathsf{h}$, which leads to a symmetric solution of \autoref{prob:exact}:
\begin{enumerate}
 \item Calculate the geodesic curve $\mu(t)$ connecting $\mu_0$ to $\mu_1$ using~\eqref{eq:fisher_rao_geodesics}.
 \item Construct a background metric $\mathsf{h}$ that is compatible with the midpoint of the geodesic $\mu(\frac12)$ using \autoref{conformalmetric} or \autoref{flatmetric}.
 \item Solve the lifting equations for the part of the geodesic that connects the densities $\mu(\frac12)$ and $\mu_1$ with respect to the compatible metric~$\mathsf{h}$. 
 Denote the resulting deformation~$\ph$.
 \item Solve the lifting equations for the part of the geodesic that connects the densities $\mu(\frac12)$ and $\mu_0$ with respect to the compatible metric~$\mathsf{h}$. 
 Denote the resulting deformation~$\psi$.
 \item The solution is then given by $\ph\circ\psi^{-1}$.
\end{enumerate}
The symmetry of this strategy follows from the construction.

\section{Relation between Fisher--Rao and other distances}
\label{sec:relation_between_the_wasserstein_and_fisher_rao_distances}

In this section we compare the Fisher--Rao distance $\distF(\cdot,\cdot)$ with the Wasserstein distance $d^{\on{W}}(\cdot,\cdot)$, the Hellinger distance $d^{\on{H}}(\cdot,\cdot)$, the total variation distance $d^{\on{TV}}(\cdot,\cdot)$, the Kullback--Leibler divergence $d^{\on{KL}}(\cdot,\cdot)$ and the $\chi^2$-distance $d^{\chi}(\cdot,\cdot)$ (see \cite{GiSu2002} for definitions).
Note that the Kullback-Leibler divergence and the $\chi^2$-distance are not metrics, as they are not symmetric.

The following inequalities for probability distances are given by \citet{GiSu2002}.
\begin{enumerate}
 \item $\frac{d^{\on{W}}(\mu_0,\mu_1)}{\on{sup}\left(d^M(x,y)\right)}\leq d^{\on{H}}(\mu_0,\mu_1)\leq \on{inf}\left( d^M(x,y)\right)\, d^{\on{W}}(\mu_0,\mu_1)$ \label{WvsH}
 \item $d^{\on{TV}}(\mu_0,\mu_1)\leq d^{\on{H}}(\mu_0,\mu_1)  \leq  \sqrt{d^{\on{TV}}(\mu_0,\mu_1)}$\label{HvsTV}
 \item $d^{\on{H}}(\mu_0,\mu_1)\leq \sqrt{d^{\on{KL}}(\mu_0,\mu_1)}$\label{HvsKL}
 \item $d^{\on{H}}(\mu_0,\mu_1)\leq \sqrt{d^{\on{\chi}}(\mu_0,\mu_1)}$\label{HvsChi}
\end{enumerate}
Here, $d^{M}$ is the distance on $M$.
The lower bound in \autoref{WvsH} is meaningful only for bounded $M$ and the upper bound only for discrete $M$.
Using these inequalities, and a comparison of the Fisher-Rao distance with the Hellinger distance, we obtain the following result.

\begin{proposition}
	For any two densities $\mu_0$, $\mu_1\in \Dens(M)$ we have
	\begin{enumerate}
		 \setcounter{enumi}{4}
		 \item $d^{\on{H}}(\mu_0,\mu_1)\leq d^{\on{FR}}(\mu_0,\mu_1)\leq \frac{\pi}{2} d^{\on{H}}(\mu_0,\mu_1)$\label{WvsFR}
		 \item $\frac{d^{\on{W}}(\mu_0,\mu_1)}{\on{sup}\left(d^M(x,y)\right)}\leq d^{\on{FR}}(\mu_0,\mu_1)\leq \frac{\pi}2\on{inf}\left( d^M(x,y)\right)\, d^{\on{W}}(\mu_0,\mu_1)$. 
		 \item $d^{\on{TV}}(\mu_0,\mu_1)\leq d^{\on{FR}}(\mu_0,\mu_1)  \leq  \sqrt{\frac{\pi}{2} d^{\on{TV}}(\mu_0,\mu_1)}$
		 \item $d^{\on{FR}}(\mu_0,\mu_1)\leq \sqrt{\frac{\pi}{2} d^{\on{KL}}(\mu_0,\mu_1)}$
		 \item $d^{\on{FR}}(\mu_0,\mu_1)\leq \sqrt{\frac{\pi}{2} d^{\on{\chi}}(\mu_0,\mu_1)}$
	\end{enumerate}
\end{proposition}
\begin{proof}
	To prove the first inequality we recall that  $d^{\on{FR}}(\mu_0,\mu_1)$ equals the spherical Hellinger distance, see \autoref{sub:fisher_metric}. Thus we only have to compare the spherical distance to the Euclidean distance, yielding a factor $\frac{\pi}{2}$. 
	Now the other inequalities follow immediately using \autoref{WvsH} to \autoref{HvsChi}.
\end{proof}

Recall from \autoref{sec:setting} that the Fisher-Rao metric arises as projection from a metric on $\Diff(M)$.
Likewise, there is a metric on $\Dens(M)$ corresponding to the Wasserstein distance that is the projection of a (non-right invariant)
metric on the diffeomorphism group~\cite{Ot2001}. 
To our knowledge, it is not known whether similar statements holds for the other distance functions. 

\bibliographystyle{amsplainnat}
\bibliography{density_matching}

\def\cprime{$'$}
\begin{thebibliography}{40}
\providecommand{\natexlab}[1]{#1}
\providecommand{\url}[1]{\texttt{#1}}
\providecommand{\urlprefix}{URL }
\providecommand{\eprint}[2][]{\url{#2}}

\bibitem[{Amari and Nagaoka(2000)}]{AmNa2000}
S.~Amari and H.~Nagaoka, \emph{Methods of information geometry}, vol. 191 of
  \emph{Translations of Mathematical Monographs}, American Mathematical
  Society, Providence, RI, 2000.

\bibitem[{Arnold and Khesin(1998)}]{ArKh1998}
V.~I. Arnold and B.~A. Khesin, \emph{Topological Methods in Hydrodynamics},
  vol. 125 of \emph{Applied Mathematical Sciences}, Springer-Verlag, New York,
  1998.

\bibitem[{Bauer et~al.(2014{\natexlab{a}})Bauer, Bruveris, and
  Michor}]{BaBrMi2014}
M.~Bauer, M.~Bruveris, and P.~Michor, Homogeneous sobolev metric of order one
  on diffeomorphism groups on real line, \emph{Journal of Nonlinear Science}
  \textbf{24} (2014{\natexlab{a}}), 769--808.

\bibitem[{Bauer et~al.(2014{\natexlab{b}})Bauer, Bruveris, and
  Michor}]{BaBrMi2015}
M.~Bauer, M.~Bruveris, and P.~Michor, {Uniqueness of the Fisher--Rao metric on
  the space of smooth densities}, {arXiv:1411.5577}, 2014{\natexlab{b}}.

\bibitem[{Bauer et~al.(2015)Bauer, Escher, and Kolev}]{BEK2014}
M.~Bauer, J.~Escher, and B.~Kolev, Local and global well-posedness of the
  fractional order {EPDiff} equation, \emph{Journal of Differential Equations}
  \textbf{258} (2015), 2010--2053.

\bibitem[{Beg et~al.(2005)Beg, Miller, Trouv{\'e}, and Younes}]{BeMiTrYo2005}
M.~F. Beg, M.~I. Miller, A.~Trouv{\'e}, and L.~Younes, Computing large
  deformation metric mappings via geodesic flows of diffeomorphisms,
  \emph{International Journal of Computer Vision} \textbf{61} (2005), 139--157.

\bibitem[{Brenier(1991)}]{Br1991}
Y.~Brenier, Polar factorization and monotone rearrangement of vector-valued
  functions, \emph{Comm. Pure Appl. Math.} \textbf{44} (1991), 375--417.

\bibitem[{Bruveris and Vialard(2014)}]{BrVi2014}
M.~Bruveris and F.~Vialard, {{O}n {C}ompleteness of {G}roups of
  {D}iffeomorphisms}, {arXiv:1403.2089}, 2014.

\bibitem[{Budd et~al.(2009)Budd, Huang, and Russell}]{BuHuRu2009}
C.~J. Budd, W.~Huang, and R.~D. Russell, Adaptivity with moving grids,
  \emph{Acta Numerica} \textbf{18} (2009), 111--241.

\bibitem[{Dominitz and Tannenbaum(2010)}]{DoTa2010}
A.~Dominitz and A.~Tannenbaum, Texture mapping via optimal mass transport,
  \emph{IEEE Transactions on Visualization and Computer Graphics} \textbf{16}
  (2010), 419--433.

\bibitem[{Dziuk and Elliott(2013)}]{DzEl2013}
G.~Dziuk and C.~M. Elliott, Finite element methods for surface pdes, \emph{Acta
  Numerica} \textbf{22} (2013), 289--396.

\bibitem[{Ebin and Marsden(1970)}]{EbMa1970}
D.~G. Ebin and J.~E. Marsden, Groups of diffeomorphisms and the notion of an
  incompressible fluid., \emph{Ann. of Math.} \textbf{92} (1970), 102--163.

\bibitem[{Fisher(1922)}]{Fi1922}
R.~A. Fisher, On the mathematical foundations of theoretical statistics,
  \emph{Philosophical Transactions of the Royal Society of London. Series A,
  Containing Papers of a Mathematical or Physical Character}  (1922), 309--368.

\bibitem[{Friedrich(1991)}]{Fr1991}
T.~Friedrich, Die fisher-information und symplektische strukturen, \emph{Math.
  Nachr.} \textbf{153} (1991), 273--296.

\bibitem[{Gibbs and Su(2002)}]{GiSu2002}
A.~L. Gibbs and F.~E. Su, On choosing and bounding probability metrics,
  \emph{Internat. Statist. Rev.}  (2002), 419--435.

\bibitem[{Gorbunova et~al.(2012)Gorbunova, Sporring, Lo, Loeve, Tiddens,
  Nielsen, Dirksen, and de~Bruijne}]{Gorbunova2012786}
V.~Gorbunova, J.~Sporring, P.~Lo, M.~Loeve, H.~A. Tiddens, M.~Nielsen,
  A.~Dirksen, and M.~de~Bruijne, Mass preserving image registration for lung
  \{CT\}, \emph{Medical Image Analysis} \textbf{16} (2012), 786 -- 795.

\bibitem[{Haker et~al.(2004)Haker, Zhu, Tannenbaum, and
  Angenent}]{HaZhTaAn2004}
S.~Haker, L.~Zhu, A.~Tannenbaum, and S.~Angenent, Optimal mass transport for
  registration and warping, \emph{Int. J. Comput. Vis.} \textbf{60} (2004),
  225--240.

\bibitem[{Hamilton(1982)}]{Ha1982}
R.~S. Hamilton, The inverse function theorem of {N}ash and {M}oser, \emph{Bull.
  Amer. Math. Soc. (N.S.)} \textbf{7} (1982), 65--222.

\bibitem[{Holm et~al.(2009)Holm, Trouv{\'e}, and Younes}]{HoTrYo2009}
D.~D. Holm, A.~Trouv{\'e}, and L.~Younes, The {E}uler-{P}oincar\'e theory of
  metamorphosis, \emph{Quart. Appl. Math.} \textbf{67} (2009), 661--685.

\bibitem[{Joshi and Miller(2000)}]{JoMi2000}
S.~Joshi and M.~Miller, Landmark matching via large deformation
  diffeomorphisms, \emph{Image Processing, IEEE Transactions on} \textbf{9}
  (2000), 1357 --1370.

\bibitem[{Khesin et~al.(2013)Khesin, Lenells, Misio{\l}ek, and
  Preston}]{KhLeMiPr2013}
B.~Khesin, J.~Lenells, G.~Misio{\l}ek, and S.~C. Preston, Geometry of
  {D}iffeomorphism {G}roups, {C}omplete integrability and {G}eometric
  statistics, \emph{Geom. Funct. Anal.} \textbf{23} (2013), 334--366.

\bibitem[{Khesin and Wendt(2009)}]{KhWe2009}
B.~Khesin and R.~Wendt, \emph{The Geometry of Infinite-dimensional Groups},
  vol.~51 of \emph{A Series of Modern Surveys in Mathematics}, Springer-Verlag,
  Berlin, 2009.

\bibitem[{Michor and Mumford(2013)}]{MiMu2013}
P.~Michor and D.~Mumford, {{O}n {E}uler's equation and '{EPD}iff'}, \emph{The
  Journal of Geometric Mechanics} \textbf{5} (2013), 319--344.

\bibitem[{Miller and Younes(2001)}]{MiYo2001}
M.~I. Miller and L.~Younes, Group actions, homeomorphisms, and matching: A
  general framework, \emph{International Journal of Computer Vision}
  \textbf{41} (2001), 61--84.

\bibitem[{Modin(2014)}]{Mo2014}
K.~Modin, Generalized {H}unter--{S}axton equations, optimal information
  transport, and factorisation of diffeomorphisms, \emph{J. Geom. Anal.}
  (2014), \eprint{http://dx.doi.org/10.1007/s12220-014-9469-2}.

\bibitem[{Modin et~al.(2011)Modin, Perlmutter, Marsland, and
  McLachlan}]{MoPeMaMc2011}
K.~Modin, M.~Perlmutter, S.~Marsland, and R.~I. McLachlan, On {E}uler-{A}rnold
  equations and totally geodesic subgroups, \emph{J. Geom. Phys.} \textbf{61}
  (2011), 1446--1461.

\bibitem[{Moselhy and Marzouk(2012)}]{MoMa2012}
T.~A.~E. Moselhy and Y.~M. Marzouk, Bayesian inference with optimal maps,
  \emph{Journal of Computational Physics} \textbf{231} (2012), 7815 -- 7850.

\bibitem[{Moser(1965)}]{Mo1965}
J.~Moser, On the volume elements on a manifold, \emph{Trans. Amer. Math. Soc.}
  \textbf{120} (1965), 286--294.

\bibitem[{Otto(2001)}]{Ot2001}
F.~Otto, The geometry of dissipative evolution equations: the porous medium
  equation, \emph{Comm. Partial Differential Equations} \textbf{26} (2001),
  101--174.

\bibitem[{Papadakis et~al.(2014)Papadakis, Peyré, and Oudet}]{PaPeOu2014}
N.~Papadakis, G.~Peyré, and E.~Oudet, Optimal transport with proximal
  splitting, \emph{SIAM Journal on Imaging Sciences} \textbf{7} (2014),
  212--238, \eprint{http://dx.doi.org/10.1137/130920058}.

\bibitem[{Rabin et~al.(2012)Rabin, Peyré, Delon, and Bernot}]{RaPeDeBe2012}
J.~Rabin, G.~Peyré, J.~Delon, and M.~Bernot, Wasserstein barycenter and its
  application to texture mixing, \emph{Scale Space and Variational Methods in
  Computer Vision}, vol. 6667 of \emph{Lecture Notes in Computer Science}, pp.
  435--446, Springer Berlin Heidelberg, 2012.

\bibitem[{Reich(2013)}]{Re2013}
S.~Reich, A nonparametric ensemble transform method for {B}ayesian inference,
  \emph{SIAM Journal on Scientific Computing} \textbf{35} (2013), A2013--A2024.

\bibitem[{Trouv{\'e} and Younes(2005)}]{TrYo2005b}
A.~Trouv{\'e} and L.~Younes, {{L}ocal geometry of deformable templates},
  \emph{SIAM J. Math. Anal.} \textbf{37} (2005), 17--59 (electronic).

\bibitem[{Trouvé and Younes(2005)}]{TrYo2005}
A.~Trouvé and L.~Younes, Metamorphoses through lie group action,
  \emph{Foundations of Computational Mathematics} \textbf{5} (2005), 173--198.

\bibitem[{ur~Rehman et~al.(2009)ur~Rehman, Haber, Pryor, Melonakos, and
  Tannenbaum}]{ReHaPrMeTa2009}
T.~ur~Rehman, E.~Haber, G.~Pryor, J.~Melonakos, and A.~Tannenbaum, 3d nonrigid
  registration via optimal mass transport on the gpu., \emph{Medical image
  analysis} \textbf{13} (2009), 931--940.

\bibitem[{Villani(2009)}]{Vi2009}
C.~Villani, \emph{Optimal transport: old and new}, vol. 338 of
  \emph{Grundlehren der Mathematischen Wissenschaften}, Springer-Verlag,
  Berlin, 2009.

\bibitem[{Yin et~al.(2009)Yin, Hoffman, and Lin}]{YiHoLi2009}
Y.~Yin, E.~A. Hoffman, and C.-L. Lin, Mass preserving nonrigid registration of
  {CT} lung images using cubic {B}-spline, \emph{Medical Physics} \textbf{36}
  (2009), 4213--4222.

\bibitem[{Younes(2010)}]{Yo2010}
L.~Younes, \emph{Shapes and Diffeomorphisms}, vol. 171 of \emph{Applied
  Mathematical Sciences}, Springer-Verlag, Berlin, 2010.

\bibitem[{Zhao et~al.(2013)Zhao, Su, Gu, Kaufman, Sun, Gao, and
  Luo}]{Zh_et_al_2013}
X.~Zhao, Z.~Su, X.~D. Gu, A.~Kaufman, J.~Sun, J.~Gao, and F.~Luo,
  Area-preservation mapping using optimal mass transport, \emph{IEEE
  Transactions on Visualization and Computer Graphics} \textbf{19} (2013),
  2838--2847.

\bibitem[{Zhu et~al.(2007)Zhu, Yang, Haker, and Tannenbaum}]{ZhYa_et_al_2007}
L.~Zhu, Y.~Yang, S.~Haker, and A.~Tannenbaum, An image morphing technique based
  on optimal mass preserving mapping, \emph{IEEE transactions on image
  processing : a publication of the IEEE Signal Processing Society} \textbf{16}
  (2007), 1481---1495.

\end{thebibliography}

\end{document}